\documentclass[11p]{article}
\usepackage[a4paper, total={6.3in, 9in}]{geometry}

\usepackage{bibspacing}
\usepackage{setspace}
\usepackage{amsmath}
\usepackage{amssymb}
\usepackage{mathtools}
\usepackage{bbm}
\usepackage{bm}
\usepackage{subcaption}
\usepackage{tikz}
\usepackage{hyperref}
\usepackage{amsthm}
\newtheorem{theorem}{Theorem}
\newtheorem{lemma}{Lemma}
\newtheorem{proposition}{Proposition}
\newtheorem{corollary}{Corollary}
\newtheorem{definition}{Definition}

\theoremstyle{definition}

\newtheorem{example}{Example}
\allowdisplaybreaks
\usepackage{algpseudocode}
\usepackage{algorithm}
\algrenewcommand\algorithmicrequire{\textbf{Input:}}
\algrenewcommand\algorithmicensure{\textbf{Output:}}
\usepackage{url}
\usepackage{kbordermatrix}
\usepackage{comment}
\usepackage[colorinlistoftodos]{todonotes}
\usepackage{cleveref}
\usepackage{booktabs}
\usepackage{multirow}
\usepackage{relsize}
\usepackage{siunitx} 
\usepackage{upgreek}

\makeatletter
\newcommand{\leqnomode}{\tagsleft@true}
\newcommand{\reqnomode}{\tagsleft@false}
\makeatother

\newcommand{\tr}{\textup{tr}}
\newcommand{\Diag}{\textup{Diag}}
\newcommand{\diag}{\textup{diag}}

\newcommand{\Conv}{\textup{conv}}

\newcommand{\rank}{\textup{rank}}

\begin{document}

\title{On Integrality in Semidefinite Programming for Discrete Optimization}
\author{Frank de Meijer \thanks{Delft Institute of Applied Mathematics, Delft University of Technology, The Netherlands, {\tt f.j.j.demeijer@tudelft.nl}}
	\and {Renata Sotirov}  \thanks{CentER, Department of Econometrics and OR, Tilburg University, The Netherlands, {\tt r.sotirov@uvt.nl}}}

\date{}

\maketitle

\begin{abstract}
    It is well-known that by adding integrality constraints to the semidefinite programming (SDP) relaxation of the max-cut problem, the resulting integer semidefinite program is an exact formulation of the problem. In this paper we show similar results for a wide variety of discrete optimization problems for which SDP relaxations have been derived. Based on a comprehensive study on discrete positive semidefinite matrices, we \textcolor{black}{introduce} a generic approach to derive mixed-integer semidefinite programming (MISDP) formulations of binary quadratically constrained quadratic programs and binary quadratic matrix programs.
Applying a problem-specific approach, we derive more compact MISDP formulations of several problems, such as the quadratic assignment problem, the graph partition problem and the integer matrix completion problem.
We also show that several structured problems allow for novel compact MISDP formulations through the notion of association schemes. Complementary to the recent advances on algorithmic aspects related to MISDP, this work opens new perspectives on solution approaches for the here considered problems. 
\end{abstract}


\section{Introduction}
Semidefinite programming (SDP) deals with the optimization of a linear function over the cone of positive semidefinite matrices under the presence of affine constraints. Over the last decades, semidefinite programs (SDPs) have proven themselves particularly useful in providing tight relaxations of discrete optimization problems. Following the extension from linear programming to integer linear programming initiated in the~1960s, a recent interest in incorporating integer variables in SDPs has arisen. Indeed, many real-world decision problems are most naturally modeled by including integer variables in optimization problems. When the variables in an SDP are required to be integer, we refer to the problem as an integer semidefinite program (ISDP). When an SDP contains both integer and continuous variables, we refer to the program as a mixed-integer semidefinite program (MISDP). As mixed-integer linear programs (MILPs) form a subclass of MISDPs, mixed-integer semidefinite programming is in general $\mathcal{NP}$-hard. 

The combination of positive semidefiniteness and integrality induces a lot of structure in matrices. Exploiting this fact, it has been shown that several structured discrete optimization problems allow for a formulation as a (M)ISDP.    
To the best of our knowledge, the first ISDP formulation of a discrete optimization problem is derived for the symmetric traveling salesman problem by Cvetkovi\'c et al.~\cite{Cvetkovic}. Eisenbl\"atter~\cite{Eissenblat} derives an ISDP formulation of the minimum $k$-partition problem, which asks for a partition of the vertex set of a graph into at most $k$ sets such that the total weight of the edges within the same set is minimized. Anjos and Wolkowicz~\cite{AnjosWolkowicz} show that the standard SDP relaxation of the max-cut problem becomes exact when adding integrality constraints. As an immediate consequence, also the SDP relaxation of the max-2-sat problem, i.e., the maximum satisfiability problem where each clause has at most two literals, see e.g.,~\cite{GoemansWilliamson},  can be modeled as an ISDP. An ISDP formulation of the chromatic number of a graph is derived by~Meurdesoif~\cite{Meurdesoif2005}. The quadratic traveling salesman problem (QTSP) is formulated as an ISDP in~\cite{deMeijerSotirovCG}. Next to these classical textbook problems, integrality in SDPs has also been at consideration in more applied problems.  Yonekura and Kanno~\cite{YonekuraKanno} formulate an optimization problem in robust truss topology design as a MISDP, see also~\cite{CerveiraEtAl, Kocvara2010,Mars}.  
The problem of computing restricted isometry constants also allows for a MISDP formulation, see~\cite{GallyPfetsch}. A MISDP formulation of the regularized cardinality constrained least squares problem is derived in~\cite{PWE:15}. Gil-Gonz\'alez et al.~\cite{GilGonzalezEtAl} use a MISDP to formulate an optimal location problem in power system analysis. Zheng et al.~\cite{ZhengEtAl} model a robust version of a power system unit commitment problem using a mixture of semidefinite constraints and integer variables. Duarte~\cite{Duarte} exploits MISDPs to find exact optimal designs of experiments in the domain of surface response modeling in statistics. Finally, Wei et al.~\cite{wei2022convex} show that mixed-integer quadratic programs with indicator variables can be solved as a MISDP.

Despite the literature on these particular problems, a generic approach for deriving problem formulations based on mixed-integer semidefinite programming has not been followed. 
Although there do exist several approaches in the literature where SDP relaxations are used in a branching scheme, see e.g.,~\cite{KrislockEtAl, RendlEtAl}, the branching strategies are based on the problem structure rather than on the matrix variables being integer.
Accordingly, exploiting integrality in the MISDP models itself has not been the method of choice so far.
This might be due to the fact that solving SDPs of large sizes is still practically challenging, discouraging to look into the extension of adding hard integrality constraints to the model. 

In this paper we refute these objections to consider MISDPs as a general solution technique, advocating that they have a great potential to be also numerically advantageous. We particularly focus on binary quadratic programs (BQPs), which aim to optimize a quadratic objective function~$g(x)$ over a feasible set~$\mathcal{X}$ defined by quadratic or linear constraints, where $x$ is required to be binary, see Figure~\ref{Fig:OverviewBQP}.  A common approach to solve these programs is by exploiting standard linearization techniques to model them as a MILP. This is often done in a branch-and-bound setting, where the subproblems correspond to the linear programming relaxations of the MILP. This research line is depicted in the top stream of Figure~\ref{Fig:OverviewBQP}. An alternative approach is to lift the vector variable $x$ in a BQP to a matrix variable $X = xx^\top$ so as to model the problem as an SDP with a nonconvex rank constraint. After relaxing this rank constraint, we obtain an SDP relaxation of the problem, see e.g.,~\cite{AnjosHandbook}. This relaxation approach corresponds to the bottom arrow in Figure~\ref{Fig:OverviewBQP}. Apart from the particular problems mentioned earlier, it is disregarded up to this point that this relaxation can also be obtained via relaxing integrality in a MISDP model that is equivalent to the BQP. More precisely, there exists a bijection between the elements in $\mathcal{X}$ and the integer points in the feasible set of the SDP relaxation. Realizing that fact, this provides a systematic way of approaching BQPs via mixed-integer semidefinite programming. Comparing the three equivalent formulations given in Figure~\ref{Fig:OverviewBQP}, the MISDP formulation has the advantages to have both a linear objective function (compared to the BQP formulation) and a convex relaxation that is often stronger than standard linear programming relaxations. After reformulating the mixed-integer nonlinear program as a convex mixed-integer nonlinear program possessing a tight relaxation, all solution techniques from convex mixed-integer nonlinear programming can be applied to tackle the problem. 
With the advancing state of the solution approaches in this field, the perspectives of this solution approach are hopeful. 

There exist various related works on the representation of optimization problems as convex mixed-integer nonlinear programs in the literature. Lubin et al.~\cite{LubinEtAl} focus on characterizations of general mixed-integer convex representability. The special case of mixed-integer second-order cone programming has received some more attention, see the survey by Benson and Sa{\u{g}}lam~\cite{BensonSaglam}. To the best of our knowledge, we are the first to address the case of mixed-integer semidefinite representability on a generic level. 

The focus of this paper is primarily on the modeling aspect of discrete optimization problems as (M)ISDPs, and less on the algorithmic aspects of solving these. With respect to the computational side, several general-purpose solution approaches have been considered recently. Gally et al.~\cite{GallyEtAl} propose a branch-and-bound framework for solving MISDPs, with the characteristic that strict duality is maintained throughout the branching tree. Solver ingredients, such as dual fixing and branching rules, are also considered in~\cite{GallyEtAl}. Kobayashi and Takano~\cite{KobayashiTakano} propose a cutting-plane and a branch-and-cut algorithm for solving generic MISDPs, where it is shown that the branch-and-cut algorithm performs best. This branch-and-cut algorithm is strengthened in~\cite{deMeijerSotirovCG}, where specialized cuts, such as Chv\'atal-Gomory cuts, are incorporated in the approach. Presolving techniques for MISDPs have been studied by Matter and Pfetsch~\cite{MatterPfetsch}. Hojny and Pfetsch~\cite{HojnyPfetsch} consider reduction techniques for solving MISDPs based on permutation symmetries. Another project that supports solving MISDPs while exploiting sparsity is GravitySDP~\cite{GravitySDP}. 
The computational ingredients of the above-mentioned approaches combined with the theoretical framework of modeling problems as (M)ISDPs that we derive in this paper, provide a complementary foundation of mixed-integer semidefinite programming in discrete optimization. 

\begin{figure}[t!]
\centering
    \tikzset{every picture/.style={line width=0.75pt}} 
\scalebox{0.55}{
\begin{tikzpicture}[x=0.75pt,y=0.75pt,yscale=-1,xscale=1]

\draw  [fill={rgb, 255:red, 198; green, 198; blue, 198 }  ,fill opacity=0.57 ] (472,1502.5) -- (621.33,1502.5) -- (621.33,1622.5) -- (472,1622.5) -- cycle ;
\draw  [fill={rgb, 255:red, 198; green, 198; blue, 198 }  ,fill opacity=0.57 ] (472,1769.5) -- (621.33,1769.5) -- (621.33,1889.5) -- (472,1889.5) -- cycle ;
\draw  [fill={rgb, 255:red, 74; green, 144; blue, 226 }  ,fill opacity=0.33 ][dash pattern={on 0.84pt off 2.51pt}] (19.33,1452.64) .. controls (19.33,1429.99) and (37.69,1411.64) .. (60.33,1411.64) -- (183.33,1411.64) .. controls (205.98,1411.64) and (224.33,1429.99) .. (224.33,1452.64) -- (224.33,1900.38) .. controls (224.33,1923.02) and (205.98,1941.38) .. (183.33,1941.38) -- (60.33,1941.38) .. controls (37.69,1941.38) and (19.33,1923.02) .. (19.33,1900.38) -- cycle ;
\draw  [fill={rgb, 255:red, 184; green, 233; blue, 134 }  ,fill opacity=1 ] (259.33,1735.5) -- (405.33,1735.5) -- (405.33,1919.3) -- (259.33,1919.3) -- cycle ;
\draw   (257.33,1474.5) -- (403.33,1474.5) -- (403.33,1658.3) -- (257.33,1658.3) -- cycle ;
\draw   (46.83,1607.09) -- (192.83,1607.09) -- (192.83,1790.9) -- (46.83,1790.9) -- cycle ;
\draw [line width=2.25]    (117.33,1793.97) -- (117.33,1964.97) -- (550.33,1964.97) -- (550.33,1901.97) ;
\draw [shift={(550.33,1896.97)}, rotate = 90] [fill={rgb, 255:red, 0; green, 0; blue, 0 }  ][line width=0.08]  [draw opacity=0] (15.72,-7.55) -- (0,0) -- (15.72,7.55) -- cycle    ;

\draw  [fill={rgb, 255:red, 184; green, 233; blue, 134 }  ,fill opacity=0.41 ][dash pattern={on 0.84pt off 2.51pt}] (224.33,1454.24) .. controls (224.33,1431.26) and (242.96,1412.64) .. (265.93,1412.64) -- (390.73,1412.64) .. controls (413.71,1412.64) and (432.33,1431.26) .. (432.33,1454.24) -- (432.33,1899.18) .. controls (432.33,1922.15) and (413.71,1940.78) .. (390.73,1940.78) -- (265.93,1940.78) .. controls (242.96,1940.78) and (224.33,1922.15) .. (224.33,1899.18) -- cycle ;
\draw [line width=2.25]    (412,1827.3) -- (462.33,1827.3) ;
\draw [shift={(467.33,1827.3)}, rotate = 180] [fill={rgb, 255:red, 0; green, 0; blue, 0 }  ][line width=0.08]  [draw opacity=0] (15.72,-7.55) -- (0,0) -- (15.72,7.55) -- cycle    ;
\draw [line width=2.25]    (412,1562.3) -- (462.33,1562.3) ;
\draw [shift={(467.33,1562.3)}, rotate = 180] [fill={rgb, 255:red, 0; green, 0; blue, 0 }  ][line width=0.08]  [draw opacity=0] (15.72,-7.55) -- (0,0) -- (15.72,7.55) -- cycle    ;
\draw [line width=0.75]    (205,1622.8) -- (246.33,1622.8)(205,1625.8) -- (246.33,1625.8) ;
\draw [shift={(253.33,1624.3)}, rotate = 180] [color={rgb, 255:red, 0; green, 0; blue, 0 }  ][line width=0.75]    (13.12,-5.88) .. controls (8.34,-2.76) and (3.97,-0.8) .. (0,0) .. controls (3.97,0.8) and (8.34,2.76) .. (13.12,5.88)   ;
\draw [shift={(198,1624.3)}, rotate = 0] [color={rgb, 255:red, 0; green, 0; blue, 0 }  ][line width=0.75]    (13.12,-5.88) .. controls (8.34,-2.76) and (3.97,-0.8) .. (0,0) .. controls (3.97,0.8) and (8.34,2.76) .. (13.12,5.88)   ;
\draw [line width=0.75]    (205,1769.8) -- (246.33,1769.8)(205,1772.8) -- (246.33,1772.8) ;
\draw [shift={(253.33,1771.3)}, rotate = 180] [color={rgb, 255:red, 0; green, 0; blue, 0 }  ][line width=0.75]    (13.12,-5.88) .. controls (8.34,-2.76) and (3.97,-0.8) .. (0,0) .. controls (3.97,0.8) and (8.34,2.76) .. (13.12,5.88)   ;
\draw [shift={(198,1771.3)}, rotate = 0] [color={rgb, 255:red, 0; green, 0; blue, 0 }  ][line width=0.75]    (13.12,-5.88) .. controls (8.34,-2.76) and (3.97,-0.8) .. (0,0) .. controls (3.97,0.8) and (8.34,2.76) .. (13.12,5.88)   ;
\draw [line width=0.75]    (332.83,1678.64) -- (332.83,1717.64)(329.83,1678.64) -- (329.83,1717.64) ;
\draw [shift={(331.33,1724.64)}, rotate = 270] [color={rgb, 255:red, 0; green, 0; blue, 0 }  ][line width=0.75]    (13.12,-5.88) .. controls (8.34,-2.76) and (3.97,-0.8) .. (0,0) .. controls (3.97,0.8) and (8.34,2.76) .. (13.12,5.88)   ;
\draw [shift={(331.33,1671.64)}, rotate = 90] [color={rgb, 255:red, 0; green, 0; blue, 0 }  ][line width=0.75]    (13.12,-5.88) .. controls (8.34,-2.76) and (3.97,-0.8) .. (0,0) .. controls (3.97,0.8) and (8.34,2.76) .. (13.12,5.88)   ;

\draw (65,1616.5) node [anchor=north west][inner sep=0.75pt]  [font=\normalsize] [align=left] {\begin{minipage}[lt]{79.25pt}\setlength\topsep{0pt}
\begin{center}
\textbf{Binary Quadratic }\\\textbf{Programming}
\end{center}

\end{minipage}};
\draw (84,1680.5) node [anchor=north west][inner sep=0.75pt]  [font=\small] [align=left] {$\displaystyle  \begin{array}{{>{\displaystyle}l}}
\min \ \ g( x)\\
\text{s.t.} \ \ \ \ x\ \in \mathcal{X}
\end{array}$};
\draw (284,1479.5) node [anchor=north west][inner sep=0.75pt]  [font=\small] [align=left] {\begin{minipage}[lt]{70pt}\setlength\topsep{0pt}
\begin{center}
\textbf{Mixed-Integer}\\\textbf{Linear}\\\textbf{Programming}
\end{center}

\end{minipage}};
\draw (285,1740) node [anchor=north west][inner sep=0.75pt]  [font=\normalsize] [align=left] {\begin{minipage}[lt]{75pt}\setlength\topsep{0pt}
\begin{center}
\textbf{Mixed-Integer }\\\textbf{Semidefinite }\\\textbf{Programming}
\end{center}

\end{minipage}};
\draw (501,1506.5) node [anchor=north west][inner sep=0.75pt]  [font=\normalsize] [align=left] {\begin{minipage}[lt]{62.4pt}\setlength\topsep{0pt}
\begin{center}
\textbf{Linear }\\\textbf{Programming}\\\textbf{Relaxation}
\end{center}

\end{minipage}};
\draw (501,1776.5) node [anchor=north west][inner sep=0.75pt]  [font=\normalsize] [align=left] {\begin{minipage}[lt]{62.4pt}\setlength\topsep{0pt}
\begin{center}
\textbf{Semidefinite }\\\textbf{Programming}\\\textbf{Relaxation}
\end{center}

\end{minipage}};
\draw (72,1420) node [anchor=north west][inner sep=0.75pt]  [font=\normalsize] [align=left] {\begin{minipage}[lt]{75pt}\setlength\topsep{0pt}
\begin{center}
\textbf{\textit{Mixed-Integer}}\\\textbf{\textit{Nonlinear }}\\\textbf{\textit{Programming}}
\end{center} 

\end{minipage}};
\draw (265,1420) node [anchor=north west][inner sep=0.75pt]  [font=\normalsize] [align=left] {\begin{minipage}[lt]{100pt}\setlength\topsep{0pt}
\begin{center}
\textbf{\textit{Convex Mixed-}}\\\textbf{\textit{Integer Nonlinear}}\\\textbf{\textit{Programming}}
\end{center}

\end{minipage}};
\draw (50,1733) node [anchor=north west][inner sep=0.75pt]  [font=\small] [align=left] {$\displaystyle g(x)$ quadratic function\\$\displaystyle \mathcal{X}$ \ \ \ integer quadratically \\ \ \ \ \ \ \ \ constrained set};
\draw (265,1542.5) node [anchor=north west][inner sep=0.75pt]  [font=\small] [align=left] {$\displaystyle  \begin{array}{{>{\displaystyle}l}}
\min \ \ f( x,y)\\
\text{s.t.} \ \ \ \ ( x,\ y) \ \in \mathcal{X}_{MILP}
\end{array}$};
\draw (265,1596.5) node [anchor=north west][inner sep=0.75pt]  [font=\small] [align=left] {$\displaystyle f(x,y)$ \ \ linear function\\$\displaystyle \mathcal{X}_{MILP}$ \ mixed-integer \\ \hspace{1.14cm} polyhedral set};
\draw (258,1803.5) node [anchor=north west][inner sep=0.75pt]  [font=\small] [align=left] {$\displaystyle  \begin{array}{{>{\displaystyle}l}}
\min \ \ h( x,X)\\
\text{s.t.} \ \ \ \ ( x,\ X) \ \in \mathcal{X}_{MISDP}
\end{array}$};
\draw (260,1856.5) node [anchor=north west][inner sep=0.75pt]  [font=\small] [align=left] {$\displaystyle h( x,X)$ \ \ \, linear function\\$\displaystyle \mathcal{X}_{MISDP}$ \ \ mixed-integer \\ \hspace{1.44cm} spectrahedral \\ 
\hspace{1.44cm}  set};
\draw (475,1571.5) node [anchor=north west][inner sep=0.75pt]  [font=\small] [align=left] {$\displaystyle  \begin{array}{{>{\displaystyle}l}}
\min \ \ f( x,y)\\
\text{s.t.} \ \ \ \ ( x,\ y) \ \in \overline{\mathcal{X}}_{MILP}
\end{array}$};
\draw (471,1840.5) node [anchor=north west][inner sep=0.75pt]  [font=\small] [align=left] {$\displaystyle  \begin{array}{{>{\displaystyle}l}}
\min \ \ h( x,X)\\
\text{s.t.} \ \ \ \ ( x,\ X) \ \in \overline{\mathcal{X}}_{MISDP}
\end{array}$};
\draw (247,1969.23) node [anchor=north west][inner sep=0.75pt]  [font=\small] [align=left] {lift $\displaystyle x$ to $\displaystyle X$ and relax rank-constraint};
\end{tikzpicture}}
\caption{Overview of various exact formulations of binary quadratic program and their relaxations. A double arrow ($\Longleftrightarrow$) denotes equivalence between the formulations, while a solid arrow ($\boldsymbol{\rightarrow}$) denotes that the formulation is relaxed from the former to the latter. The sets $\mathcal{X}$, $\mathcal{X}_{MILP}$ and $\mathcal{X}_{MISDP}$ are defined by nonconvex integer constraints, while $\overline{\mathcal{X}}_{MILP}$ and $\overline{\mathcal{X}}_{MISDP}$ are convex relaxations. \label{Fig:OverviewBQP}}
\end{figure}
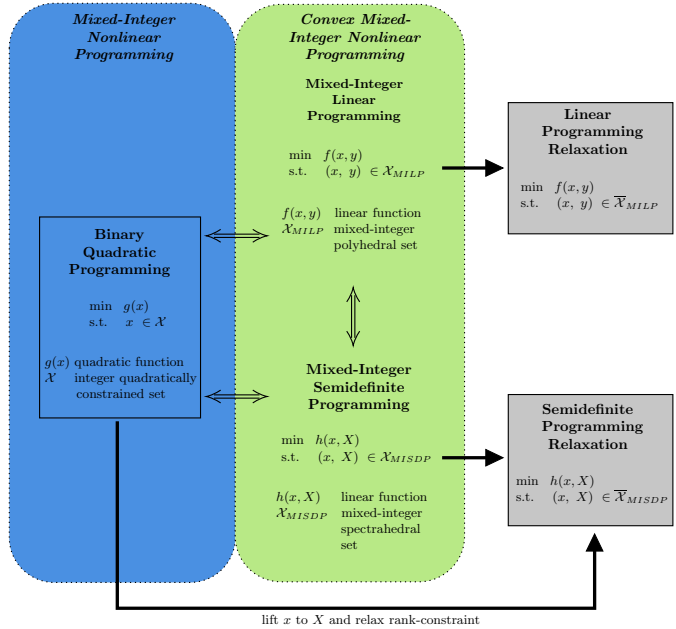

\subsection*{Main results and outline}

This paper studies the theoretical role of mixed-integer semidefinite programming in discrete optimization. We show that many problems can be modeled as a (M)ISDP, either by a generic approach for certain large problem classes, or by a more problem-specific approach. Gradually, we cover and exploit results from matrix theory, optimization, algebraic combinatorics and graph theory. Our approach is accompanied with a large number of examples of various discrete optimization problems. 

We start our approach with an extensive overview of results on the matrix theory of discrete positive semidefinite (PSD) matrices. Without considering explicit optimization problems, we focus on the structure of PSD $\{0,1\}$--, $\{\pm 1\}$-- and $\{0,\pm 1\}$--matrices. This overview reviews results from~\cite{BermanXu, DukanovicRendl, LetchfordSorensen, LiederEtAl}, but also introduces new results and formulations of these matrix sets, such as a combinatorial viewpoint of PSD~$\{0,1\}$--matrices of rank at most~$r$. We also extend results that are known for $\{0,1\}$--matrices to the other two matrix sets.

These matrix theoretical results are exploited when proving that many binary quadratic problems allow for a formulation as a binary semidefinite program (BSDP). We establish this result for binary quadratically constrained quadratic programs and, in particular, for binary quadratic matrix programs. Problems that allow for a formulation as a binary quadratic matrix program, e.g., quadratic clustering or packing problems, can be modeled as a compact BSDP with a PSD matrix variable of relatively low order. 

After that, we treat several specific problem classes for which we obtain MISDP formulations that do not follow from the above-mentioned framework or for which it is possible to obtain a more efficient formulation. Among these are the quadratic assignment problem, including its extensive number of special cases, see e.g.,~\cite{Burkard}, several graph partition problems and graph problems that can be modeled based on association schemes. Also, as most formulations that we discuss include binary variables, we present two problems that have a MISDP formulation where the variables are integer, but nonbinary, namely the integer matrix completion problem~\cite{Netflix} and the sparse integer least squares problem~\cite{PiaZhou}. 

This paper is structured as follows. In Section~\ref{Sect:TheoryPSD} we present results on the matrix theory of discrete PSD matrices. These results are exploited in Section~\ref{Sect:BQPformulations}, where MISDP formulations of generic quadratically constrained quadratic programs and quadratic matrix programs are derived. In Section~\ref{Sect:ProblemSpecific} we treat problem-specific MISDP formulations that do not follow from the previous section. Concluding remarks are given in Section~\ref{Sect:Conclusion}.

\subsection*{Notation}
We denote by $\mathbf{0}_n \in \mathbb{R}^n$  the vector of all zeros, and by $\mathbf{1}_n \in \mathbb{R}^n$ the vector of all  ones. The identity matrix and the matrix of ones of order $n$ are denoted by $\mathbf{I}_n$ and $\mathbf{J}_n$, respectively. We omit the subscripts of these matrices when there is no confusion about the order.
The $i$th {\color{black}unit vector} is denoted by $\mathbf{e}_i$ and we define $\mathbf{E}_{ij} := \mathbf{e}_i\mathbf{e}_j^\top$. 
The set of $n\times n$ permutation matrices is denoted by~$\Pi_n$.

For $n \in \mathbb{Z}_+$, we define the set $[n] := \{1, \ldots, n\}$.
For any $S \subseteq [n]$, we let $\mathbf{\mathbbm{1}_{S}}$ be the binary indicator vector of $S$.   
The support of $x\in \mathbb{R}^n $
is denoted by $\mbox{supp}(x)$.

We denote the set of all $n \times n$ real symmetric matrices by  $\mathcal{S}^n$.
The cone of symmetric positive semidefinite matrices is defined as $\mathcal{S}^n_+ := \{ {X} \in \mathcal{S}^n \, : \, \, {X} \succeq \mathbf{0} \}$, where ${X} \succeq \mathbf{0}$ denotes  that $X$ is PSD.
The trace of a square matrix ${X}=(X_{ij})$ is given by $\tr({X})=\sum_{i}X_{ii}$.
 For any ${X},{Y} \in \mathbb{R}^{n \times m}$ the trace inner product is defined as $\langle {X}, {Y} \rangle := \tr({X}^\top {Y}) = \sum_{i = 1}^n\sum_{j = 1}^m X_{ij} Y_{ij}$.

 The operator $\diag : \mathbb{R}^{n \times n} \rightarrow \mathbb{R}^n$ maps a square matrix to a vector consisting of its diagonal elements. We denote by $\Diag : \mathbb{R}^n \rightarrow \mathbb{R}^{n \times n}$ its adjoint operator.
The Hadamard product of two matrices~${X=(X_{ij}), Y=(Y_{ij}) \in  \mathbb{R}^{n \times m}}$ is denoted by $X\circ Y$
and is defined as $(X\circ Y)_{ij} := X_{ij} Y_{ij}$.
The direct sum of matrices $X$ and $Y$ is defined as $X \oplus Y = \begin{psmallmatrix} X & \mathbf{0} \\  \mathbf{0} & Y \end{psmallmatrix}.$
The Kronecker product $X \otimes Y$ of matrices $X \in \mathbb{R}^{p \times q}$ and $Y\in \mathbb{R}^{r\times s}$ is defined as the $pr \times qs$ matrix composed of $pq$ blocks of size~$r\times s$, with block $ij$ given by~$x_{ij} Y$,~$i \in [p]$,~$j \in [q]$.

\section{Theory on discrete PSD matrices} \label{Sect:TheoryPSD}

Most discrete optimization problems that we consider in this paper are defined using binary variables, i.e., variables taking values in~$\{0,1\}$ or~$\{\pm 1\}$, or ternary variables, i.e., variables whose values are in~$\{0, \pm 1\}$. In this section we derive several useful results on these matrix sets with respect to positive semidefiniteness. We start by considering the PSD~$\{0,1\}$--matrices, after which we extend these results to PSD $\{\pm 1\}$-- and~$\{0,\pm 1\}$--matrices.

\subsection[Theory on PSD $\{0,1\}$--matrices]{Theory on PSD $\mathbf{\{0,1\}}$--matrices} \label{Subsect:theory01}
In this section we consider the set of positive semidefinite $\{0,1\}$--matrices. We derive and recall several formulations of this matrix set, including a combinatorial, polyhedral and a set-completely positive description. {\color{black}We start this section with two known results, i.e., Theorem~\ref{Thm:binPSD} and Proposition~\ref{Prop:binaryProperties}, after which we present a series of new results.}

Positive semidefinite $\{0,1\}$--matrices are studied explicitly by Letchford and S{\o}rensen~\cite{LetchfordSorensen}. They derive the following decomposition result on PSD $\{0,1\}$--matrices. 
\begin{theorem}[\cite{LetchfordSorensen}]
Let ${X} \in \{0,1\}^{n \times n}$ be a symmetric matrix. Then ${X} \succeq \mathbf{0}$ if and only if ${X} = \sum_{j = 1}^r x_j x_j^\top$ for some ${x_j} \in \{0,1\}^n$, $j \in [r]$. \label{Thm:binPSD}
\end{theorem}

Given $S \subseteq \mathbb{R}_+$, a matrix $X$ is called $S$-completely positive if $X = PP^\top$ for some~$P \in S^{n \times k}$. In case~$S = \mathbb{R}_+$, we call $X$ completely positive. It follows from Theorem~\ref{Thm:binPSD} that any PSD~$\{0,1\}$--matrix is~$\{0,1\}$--completely positive, see also Berman and Xu~\cite{BermanXu}.

The decomposition of PSD $\{0,1\}$--matrices gives rise to a useful combinatorial interpretation on the complete graph $K_n$. Viewing each vector $x_j \in \{0,1\}^n$ as an indicator vector on the vertices of $K_n$, the matrix~$x_jx_j^\top$ can be seen as the characteristic matrix of a clique in~$K_n$. Given a decomposition~$X = \sum_{j=1}^k x_jx_j^\top$, the cliques indexed by $j\in [k]$ are pairwise disjoint, since the diagonal of $X$ is at most one. Therefore, each PSD $\{0,1\}$--matrix is the characteristic matrix of a set of pairwise disjoint cliques in~$K_n$. This combinatorial structure is in the literature also known as a clique packing. 

\medskip

As we will see in the next section, many SDP formulations arise from a lifting~$PP^\top$, where $P$ is an appropriate $n \times r$ $\{0,1\}$--matrix. Consequently, the resulting PSD~$\{0,1\}$--matrix has rank at most $r$. From that perspective, it makes sense to consider the set of PSD~$\{0,1\}$--matrices that have an upper bound on the rank. For positive integers~$r, n$ with~$r \leq n$, let us define the discrete set 
\begin{align}
    \mathcal{D}^n_r &:= \left\{ X \in \{0,1\}^{n \times n} \, : \, \, X \succeq \mathbf{0},~\rank(X) \leq r \right\}. \label{Def:setD}
\intertext{Theorem~\ref{Thm:binPSD} induces the following $\{0,1\}$--completely positive description of $\mathcal{D}^n_r$:}
    \mathcal{D}^n_r & \,= \left\{PP^\top \, : \, \, P \in \{0,1\}^{n \times r},~ P\mathbf{1}_r \leq \mathbf{1}_n \right\}. \label{Eq:setDbinaryCP}
\end{align}
Next, we will derive another formulation of $\mathcal{D}^n_r$, where the constraint $\rank(X) \leq r$ is established by an appropriate linear matrix inequality. 
To that end, we exploit the following result that is implicitly proved in many sources and explicitly by Dukanovic and Rendl~\cite{DukanovicRendl}.

\begin{proposition}[\cite{DukanovicRendl}]
\label{Prop:binaryProperties}
Let ${X} \in \{0,1\}^{n \times n}$ be symmetric. Then, the following are equivalent: 
\begin{itemize}
\item[(i)] $\diag({X})=\mathbf{1}_n$, $\rank({X})=r$, ${X} \succeq \mathbf{0}$.
    \item[(ii)] There exists a permutation matrix ${Q}$ such that ${QXQ}^\top=\mathbf{J}_{n_1} \oplus \, \cdots \, \oplus \mathbf{J}_{n_r}$ with $n = n_1 + \dots + n_r$.
    \item[(iii)] $\diag({X})=\mathbf{1}_n$, $\rank({X})=r$ and ${X}$ satisfies the triangle inequalities $X_{ij} + X_{ik} - X_{jk} \leq 1$ for all $(i,j,k) \in [n] \times [n] \times [n]$. 
\item[(iv)] $\diag(X) = \mathbf{1}_n$ and $(tX - \mathbf{J} \succeq \mathbf{0} \Longleftrightarrow t \geq r)$. 
\end{itemize}
\end{proposition}

Proposition~\ref{Prop:binaryProperties} establishes the equivalence between useful characterizations of rank-$r$ PSD~$\{0,1\}$--matrices that have ones on the diagonal. In the following corollary we generalize this result by relaxing the condition $\diag(X) = \mathbf{1}_n$. {\color{black}The proof is similar to the proof of~Proposition~\ref{Prop:binaryProperties}.}
\begin{corollary} \label{Cor:binaryProperties}
    Let $X \in \{0,1\}^{n \times n}$ be symmetric. Then, the following statements are equivalent:
    \begin{itemize}
    \item[(i)] $\rank({X})=r$, ${X} \succeq \mathbf{0}$.
    
    \item[(ii)] There exists a permutation matrix ${Q}$ such that ${QXQ}^\top=\mathbf{J}_{n_1} \oplus \, \cdots \oplus \mathbf{J}_{n_r} \oplus \mathbf{0}_{n_z \times n_z}$ with $n = n_1 + \dots + n_r + n_z$.
    
    \item[(iii)] $\rank({X})=r$ and ${X}$ satisfies the triangle inequalities $X_{ij} \leq X_{ii}$ for all $i \neq j$ and $X_{ij} + X_{ik} - X_{jk} \leq X_{ii} $ for all $j < k, i \neq j, k$. 
      
\item[(iv)] $tX - \diag(X)\diag(X)^\top \succeq \mathbf{0}$ if and only if $t \geq r$. 
\end{itemize}
\end{corollary}
\begin{proof}  
    Throughout the proof, let ${N_1 := \{i \in [n] \, : \, \, X_{ii} = 1\}}$ and let~$N_0 := [n] \setminus N_1$. Moreover, let $Q'$ denote a permutation matrix corresponding to a permutation of $[n]$ that maps the ordered set~$(1, 2, \ldots, n)$ to an ordered set where the elements in $N_1$ occupy the first $N_1$ positions.

    $(i) \Longleftrightarrow (ii):$  Let $X$ be positive semidefinite with $\rank(X) = r$. Then, the rows and columns indexed by $N_0$ only contain zeros. As a consequence, $Q'X(Q')^\top$ is of the form~${Y \oplus \mathbf{0}_{|N_0| \times |N_0|}}$ with~$\diag(Y) = \mathbf{1}_{|N_1|}$ and $Y \succeq \mathbf{0}$. By Proposition~\ref{Prop:binaryProperties} there exists a permutation matrix $\bar{Q}$ such that $\bar{Q}Y\bar{Q}^\top = \mathbf{J}_{n_1} \oplus \, \cdots \, \oplus \bold{J}_{n_r}$ with ${|N_1| = n_1 + \dots + n_r}$. Let~${Q := (\bar{Q} \oplus \bold{I}_{|N_0|} )Q'}$, then~$QXQ^\top = \bold{J}_{n_1} \oplus \, \cdots \oplus \bold{J}_{n_r} \oplus \mathbf{0}_{|N_0| \times |N_0|}$. 

    Conversely, suppose that ${QXQ}^\top=\bold{J}_{n_1} \oplus \, \cdots \oplus \mathbf{J}_{n_r} \oplus \mathbf{0}_{n_z \times n_z}$ with ${n = n_1 + \dots + n_r + n_z}$ for some permutation matrix ${Q}$. Then, obviously, $QXQ^\top$ is positive semidefinite with $\rank(QXQ^\top) = r$, from which it follows that $X \succeq \mathbf{0}$ with $\rank(X) = r$.

    $(i) \Longleftrightarrow (iii):$ For $n = 2$, the inequalities $X_{ij} \leq X_{ii}$, $i \neq j$, are trivially necessary and sufficient for $X \succeq \bold{0}$. For $n \geq 3$, the result follows from Letchford and S{\o}rensen \cite[Proposition~3]{LetchfordSorensen}.      
     
     ${(i) \Longleftrightarrow (iv):}$  Suppose $X$ is PSD with $\rank(X) = r$. Then,~$Q'X(Q')^\top = {Y \oplus \mathbf{0}_{|N_0| \times |N_0|}}$ with~$\diag(Y) = \mathbf{1}_{|N_1|}$ and $Y \succeq \mathbf{0}$. This leads to the following sequence of equivalences: 
     \begin{align*}
         tX - \diag(X)\diag(X)^\top \succeq \bold{0} \quad & \Longleftrightarrow \quad t \, Q'X(Q')^\top - Q'\diag(X) \diag(X)^\top (Q')^\top \succeq \bold{0} \\
         & \Longleftrightarrow \quad t\left( Y \oplus \bold{0}_{|N_0| \times |N_0|}\right) - \begin{pmatrix}
             \bold{1}_{|N_1|} \\
             \bold{0}_{|N_0|}
         \end{pmatrix} \begin{pmatrix}
             \bold{1}_{|N_1|} \\
             \bold{0}_{|N_0|}
         \end{pmatrix}^\top  \succeq \bold{0} 
     \end{align*}
     which holds if and only if $t \geq r$, by statement $(iv)$ of Proposition~\ref{Prop:binaryProperties}.   

     Conversely, suppose that $tX - \diag(X) \diag(X)^\top \succeq \bold{0}$ if and only if $t \geq r$. If $r = 0$, then~${t = 0}$ induces $\diag(X) = \bold{0}_n$, while $t = 1$ implies~$X - \diag(X)\diag(X)^\top = X \succeq \bold{0}$, {\color{black} since $\diag(X) = \bold{0}_n$}. Hence,~$X$ must be the zero matrix, which is positive semidefinite with rank zero. Now, assume that~$r \geq 1$. Then,~$rX - \diag(X) \diag(X)^\top \succeq \bold{0}$ can be written as ${X \succeq \frac{1}{r} \diag(X) \diag(X)^\top \succeq \bold{0}}$. Let $r^* := \rank(X)$. It follows from the previously proven implication, $(i) \Longrightarrow (iv)$, that~${r^* = \min\{t \, : \, \, tX - \diag(X)\diag(X)^\top \succeq \bold{0} \}}$. By assumption, this value equals~$r$, so~$r^* = r$. We conclude that $X$ is PSD with $\rank(X) = r$.  
\end{proof}
Corollary~\ref{Cor:binaryProperties} can be exploited to prove the following result.
\begin{corollary} \label{Cor:binPSDrank_upper}
    Let $X \in \{0,1\}^{n \times n}$ be symmetric. If $Y = \begin{psmallmatrix}
        r & \diag(X)^\top \\
        \diag(X) & X
    \end{psmallmatrix} \succeq \bold{0}$, then $X \succeq \bold{0}$ with $\rank(X) \leq r$. 
\end{corollary}
\begin{proof}
    The assertion $X \succeq \bold{0}$ is trivial, so it suffices to show that $Y \succeq \bold{0}$ implies $\rank(X) \leq r$. If~$r = 0$, then $\diag(X) = \bold{0}_n$. Since $X \succeq \bold{0}$, $X$ must be the zero matrix and, thus,~${\rank(X) = 0}$. 

    Now, let $r \geq 1$. 
    The Schur complement lemma implies that $rX - \diag(X)\diag(X)^\top \succeq \bold{0}$.  Let~$r^* := \min\left\{ t \, : \, \, tX - \diag(X)\diag(X)^\top \succeq \bold{0} \right\} \leq r$. 
    Since ${r^*X - \diag(X)\diag(X)^\top \succeq \bold{0}}$ and~${X \succeq \bold{0}}$, it follows that $tX - \diag(X)\diag(X)^\top \succeq \bold{0}$ for all $t \geq r^*$. Therefore, $tX - \diag(X)\diag(X)^\top \succeq \bold{0}$ if and only if $t \geq r^*$. Corollary~\ref{Cor:binaryProperties} then implies~${\rank(X) = r^* \leq r}$. 
\end{proof}
Corollary~\ref{Cor:binPSDrank_upper} implies the following characterization of $\mathcal{D}^n_r$, where the rank constraint is merged into a lifted linear matrix inequality: 
\begin{align} \label{Eq:SetDdiscreteLMI}
    \mathcal{D}^n_r = \left\{ X \in \mathcal\{0,1\}^{n \times n} \, : \, \, \begin{pmatrix}
        r & \diag(X)^\top \\ \diag(X) & X
    \end{pmatrix} \succeq \bold{0} \right\}. 
\end{align}
For some optimization problems, the upper bound constraint on the rank of $X$ is not sufficient, as we require that $X$ is {exactly} of rank $r$. The max $k$-cut problem, for instance, requires to partition the vertex set of a graph into exactly $k$ nonempty and pairwise disjoint subsets. The following two results show that the description given in~\eqref{Eq:SetDdiscreteLMI} can be extended to also include a lower bound on the rank of $X$. 
\begin{theorem} \label{Thm:binPSDrank_lower}
    Let $X \in \{0,1\}^{n \times n}$ be symmetric. If there exists a matrix $P \in \{0,1\}^{n \times r}$ with~${P^\top \bold{1} \geq \bold{1}}$ such that $Y = \begin{psmallmatrix}
        \bold{I}_r & P^\top \\ P & X
    \end{psmallmatrix} \succeq \bold{0}$, then $X \succeq \bold{0}$ with $\rank(X) \geq r$. 
\end{theorem}
\begin{proof}
    The assertion $X \succeq \bold{0}$ is trivial. It suffices to show that $\rank(X) \geq r$. As $Y \succeq \bold{0}$ and~$Y$ has binary entries, it follows from Theorem~\ref{Thm:binPSD} that
        $Y = \sum_{j =1}^k \begin{psmallmatrix}
            u_j \\ x_j
        \end{psmallmatrix} \begin{psmallmatrix}
            u_j \\ x_j
        \end{psmallmatrix}^\top$ 
    for some $u_j \in \{0,1\}^r$ and $x_j \in \{0,1\}^n$, $j \in [k]$. Since $\sum_{j = 1}^k u_ju_j^\top = \bold{I}_r$, we must have $k \geq r$. Moreover, the set~$\{u_j \, : \, \, j \in [k]\}$ must contain $\bold{e}_1, \ldots, \bold{e}_r$ and $k - r$ copies of $\bold{0}_r$. Without loss of generality, let us assume that the first $r$ vectors in $\{u_j \, : \, \, j\in [k]\}$ correspond to the elementary vectors. Then, it follows that~$P = \sum_{j = 1}^k x_ju_j^\top = \sum_{j = 1}^r x_j \bold{e}_j^\top = [x_1 ~ \dots ~x_r ]$. 
    Since $P^\top \bold{1} \geq \bold{1}$, it follows that the vectors~$x_j$, $j \in [r]$, cannot be the zero vector. Since these are moreover linearly independent, we have $\rank(X) \geq \rank(\sum_{j = 1}^r x_jx_j^\top) = r$. 
\end{proof}
Theorem~\ref{Thm:binPSDrank_lower} and Corollary~\ref{Cor:binPSDrank_upper} together impose the following integer semidefinite characterization of PSD~$\{0,1\}$--matrices of rank $r$. 
\begin{corollary}\label{Cor:ExactRank}
    Let $X \in \{0,1\}^{n \times n}$ be symmetric. If there exists a matrix $P \in \{0,1\}^{n \times r}$ with~${P^\top \bold{1} \geq \bold{1}}$, $P \bold{1} = \diag(X)$, such that $Y = \begin{psmallmatrix}
        \bold{I}_r & P^\top \\ P & X
    \end{psmallmatrix} \succeq \bold{0}$, then $X \succeq \bold{0}$ with~${\rank(X) = r}$.
\end{corollary}
\begin{proof}
    It immediately follows from Theorem~\ref{Thm:binPSDrank_lower} that ${\color{black}X} \succeq \bold{0}$ with $\rank(X) \geq r$. Moreover, since $Y \succeq \bold{0}$, we also know that
    \begin{align*}
        \left( \bold{1}_r^\top \oplus \bold{I}_n \right) Y \left( \bold{1}_r^\top \oplus \bold{I}_n \right)^\top = \begin{pmatrix}
            \bold{1}_r^\top \bold{I}_r \bold{1}_r &  \bold{1}_r^\top P^\top \bold{I}_n \\
            \bold{I}_n P \bold{1}_r &  \bold{I}_n X \bold{I}_n 
        \end{pmatrix} = \begin{pmatrix}
            r & \diag(X)^\top \\
            \diag(X) & X
        \end{pmatrix} \succeq \bold{0}. 
    \end{align*}
    It then follows from Corollary~\ref{Cor:binPSDrank_upper} that $\rank(X) \leq r$. 
\end{proof}

The integer semidefinite characterization of $\mathcal{D}^n_r$ given in~\eqref{Eq:SetDdiscreteLMI} shows that if a $\{0,1\}$--matrix satisfies a certain linear matrix inequality, then a rank condition is implied. For the case of rank-one matrices, we can show that the converse implication does also hold, i.e., if a rank-one matrix satisfies a certain linear matrix inequality, then its entries must be in $\{0,1\}$. 
\begin{theorem} \label{Thm:binPSDrank}
Let $Y=\begin{psmallmatrix}
1 & {x}^\top  \\ {x} & {X}
\end{psmallmatrix} \succeq \mathbf{0}$ with $\diag({X}) = {x}$. Then, ${\rank(Y)} = 1$ if and only if~${{X} \in \{0,1\}^{n \times n}}$. 
\end{theorem}
\begin{proof}
$(\Longrightarrow ):$ 
If $\rank(Y) = 1$, then $Y = \bar{x}\bar{x}^\top$ with $\bar{x} = [
    1~ x^\top ]^\top \in \mathbb{R}^{n+1}$ and~${X = xx^\top}$. 
From the positive semidefiniteness of order two principal submatrices of  $Y$ we obtain~${\mathbf{0}_n \leq x \leq \mathbf{1}_n}$. Since 
$\diag(xx^\top) = {x}$, we have $x_i^2 = x_i$ for all $i \in [n]$, so
 $x \in \{0,1\}^n$. We conclude that~${{X} = xx^\top \in \{0,1\}^{n \times n}}$.

$(\Longleftarrow ):$
Since ${X} \in \{0,1\}^{n \times n}$ and ${x} = \diag({X})$, it follows that $Y \in \{0,1\}^{(n+1)\times (n+1)}$. From Theorem~\ref{Thm:binPSD} it follows that $Y= \sum_{j=1}^k x_jx_j^\top$ for some ${x_j} \in \{0,1\}^{n+1}$, $j \in [k]$, i.e.,~$Y$ can be decomposed in terms of cliques. Since $Y_{11} = 1$ and $\diag(Y) = (1, {x}^\top )^\top$, all indices~$i \in [n+1]$ for which $Y_{ii} = 1$ must be in the same clique as the first index. 
Hence, the decomposition consists of only one clique and $\rank(Y) = 1$. 
\end{proof}
Theorem~\ref{Thm:binPSDrank} plays a central role in deriving integer SDP formulations of binary quadratic problems defined over vectors of variables in Section~\ref{section:vector lifting}. However,  Theorem~\ref{Thm:binPSDrank} cannot be extended to matrices with a rank larger than one. \textcolor{black}{That is, if $Y$ is a PSD matrix satisfying $\diag(Y) = Y\bold{e}_1$ and $Y_{11} = 2$, then integrality of $Y$ is not equivalent to be of rank 2.} 
For example, the matrix $Y=\frac{1}{2}( \bold{J}_3 + 3 \mathbf{E}_{11})$ satisfies $Y \succeq \bold{0}$, $\diag(Y) = Y\mathbf{e}_1$ and $\rank(Y) = 2$, but $Y$ is not integer. 

\medskip

The characterizations given in~\eqref{Eq:setDbinaryCP} and~\eqref{Eq:SetDdiscreteLMI} rely on conditions involving discreteness. Let us now move on to continuous descriptions. Of course, since $\mathcal{D}^n_r$ is itself a discrete set, a continuous description does not aim at describing $\mathcal{D}^n_r$, but rather its convex hull, i.e.,
\begin{align} 
    \mathcal{P}^n_r &:= \Conv(\mathcal{D}^n_r). \label{Def:setP}
\end{align}
Observe that although the matrices in $\mathcal{D}^n_r$ have an upper bound on the rank, the polytopes~$\mathcal{P}^n_r$ are full-dimensional, since $\frac{1}{n}\bold{I}_n \in \mathcal{P}^n_r$ for all $1 \leq r \leq n$. 
In order to gain more insight into the structure of~$\mathcal{P}^n_r$, we introduce the notion of a so-called packing family. 

\begin{definition}
    Let $T$ be a finite set of elements. A collection $\mathcal{F}$ of nonempty subsets of $T$ is called a packing of $T$ if the subsets in $\mathcal{F}$ are pairwise disjoint. The family of all packings of $T$ is called the packing family of $T$, denoted by $\bold{F}(T)$. 
\end{definition}
Observe that $\mathcal{F} = \emptyset$ also belongs to $\bold{F}(T)$. 
Next, we define the notion of an~$r$-packing of $T$.

\begin{definition} \label{Def:rpacking}
    Let $T$ be a finite set of elements. A packing $\mathcal{F}$ of $T$ is called an~$r$-packing of $T$ if~$|\mathcal{F}| \leq r$. The family of all $r$-packings of $T$ is called the~$r$-packing family of $T$, denoted by $\bold{F}_r(T)$. 
\end{definition}

The $r$-packing family of $[n]$ can be exploited to describe $\mathcal{P}^n_r$. 
Let $X \in \mathcal{D}^n_r$. By Theorem~\ref{Thm:binPSD} we know that~$X$ is the sum of at most $r$ rank-one PSD $\{0,1\}$--matrices. From a combinatorial point of view, this implies that $X$ corresponds to an $r$-packing of $[n]$. In fact, there is a bijection between the matrices in $\mathcal{D}_r^n$ and the $r$-packings in $\bold{F}_r([n])$. For any $r$-packing $\mathcal{F}$, let $\mathbf{E}_{\mathcal{F}} := \sum_{S \in \mathcal{F}} \mathbbm{1}_S \mathbbm{1}_S^\top$. Then, we obtain the following polyhedral description of $\mathcal{P}^n_r$ for all positive integers $r \leq n$:
\begin{align} \label{Eq:PackingDescription}
    \mathcal{P}^n_r = \left\{ X \in \mathcal{S}^n \, : \, \,  X = \sum_{\mathcal{F} \in \bold{F}_r([n])} \lambda_{\mathcal{F}} \bold{E}_{\mathcal{F}},  \sum_{\mathcal{F} \in \bold{F}_r([n])} \lambda_{\mathcal{F}} = 1,~ \lambda_{\mathcal{F}} \geq 0 \text{  for all  } \mathcal{F} \in \bold{F}_r([n]) \right\}. 
\end{align}
{\color{black}We call}  the description above the {\em packing description} of $\mathcal{P}^n_r$. Let us now consider the cardinality of the vertices of $\mathcal{P}^n_r$.

In the vein of Definition~\ref{Def:rpacking}, we call $\mathcal{F} \subseteq \mathbb{P}([n])$ an $r$-partition of $[n]$ if it is an $r$-packing with~$\bigcup_{S \in \mathcal{F}} = [n]$. Here, $\mathbb{P}([n])$ denotes the power set of $[n]$.
The number of partitions of the set $[n]$ into $k$ nonempty subsets is in the literature known as the Stirling number of the second kind, denoted by ${ n \brace k}$. The total number of partitions of $[n]$ equals the Bell number~$B_n$~\cite{Bell}, for which we have $B_n = \sum_{k = 0}^n {n \brace k}$.
We can now show the following result regarding the cardinality of $\mathcal{D}^n_r$. 
\begin{theorem} \label{Thm:BellStirling}
For $n \geq 1$ and $0 \leq r \leq n$, we have $|\mathcal{D}_r^n| = \sum_{k = 1}^{r+1}{ n+1 \brace k}$. In particular,~${|\mathcal{D}_1^n| = 2^n}$ and~$|\mathcal{D}_n^n| = B_{n+1}$. 
\end{theorem}
\begin{proof}
It follows from the discussion above that $|\mathcal{D}_r^n|$ equals the number of $r$-packings in~$\bold{F}_r([n])$. In order to count these, we count the number of packings that consist of exactly $k$ subsets, while $k$ ranges from 0 to $r$. \textcolor{black}{Any packing of $[n]$ into $k$ subsets corresponds to a partition of $[n+1]$ into $k + 1$ subsets. To see this, observe that to each packing $\mathcal{F}$ of $[n]$ into $k$ subsets one can add a~$(k+1)$th set containing the element $n+1$ and the elements not covered by $\mathcal{F}$. Conversely, given a partition of $[n+1]$ into $k + 1$ subsets, dropping the set containing the element $n+1$ yields a packing of $[n]$ consisting of exactly $k$ subsets.} Hence, the number of packings of $[n]$ consisting of exactly $k$ subsets equals ${ n+1 \brace k+1}$ and     
$|\mathcal{D}_r^n| = \sum_{k = 0}^r {n+1 \brace k+1} = \sum_{k = 1}^{r+1} {n+1 \brace k}$.
    For the special case $r = 1$, we obtain $|\mathcal{D}^n_1| = {n+1 \brace 1} + {n+1 \brace 2} = 1 + \frac{2^{n+1} - 2}{2} = 2^n$. When $r = n$, we exploit ${n+1 \brace 0} = 0$ to conclude that $|\mathcal{D}^n_n| = \sum_{k = 1}^{n+1} {n+1 \brace k} = \sum_{k = 0}^{n+1} {n+1 \brace k} = B_{n+1}$.
\end{proof}
The polytope $\mathcal{P}^n_r$ has several relationships with other well-known polytopes from the literature. Letchford and S{\o}rensen \cite{LetchfordSorensen} study the polytope $\mathcal{P}^n_n$, albeit in a different embedding, and refer to it as the binary PSD polytope of order $n$. They emphasize its relationship with the clique partitioning polytope that was introduced by Gr\"otschel and Wakabayashi~\cite{GrotschelWakabayashi} and later studied in \cite{BandeltEtAl, OostenEtAl}. Given the complete graph~${G = (V,E)}$, a clique partition is a subset $A \subseteq E$ such that there is a partition of $V$ into nonempty disjoint sets $V_1, \ldots, V_k$ such that each $V_j$, $j \in [k]$, induces a clique in $G$ and~${A = \bigcup_{j \in [k]} \{\{i,\ell \} \, : \, \, i, \ell \in V_j, i \neq \ell\}}$. The incidence vectors of clique partitions are only defined on the edge set, and therefore the clique partition polytope can be seen as a projection of $\mathcal{P}^n_n$. 

Among one of the first graph partition problems is the one considered by Chopra and Rao~\cite{ChopraRao}. Given an undirected graph $G$, the vertices need to be partitioned into at most $k$ subsets so as to minimize the total cost of edges cut by the partition. 
If $G$ is the complete graph, the partition polytope $P\mathit{1}(r)$ considered in~\cite{ChopraRao} coincides with $\mathcal{P}^n_r$ (apart from the embedding). 

The polytope $\mathcal{P}^n_n$ can also be related to the stable set polytope. Let~${G_{\mathbb{P}} = (V_{\mathbb{P}}, E_{\mathbb{P}})}$ be the power set graph, i.e., each vertex in $V_{\mathbb{P}}$ corresponds to a nonempty subset of~$[n]$ and the edge set is defined as~${E_{\mathbb{P}} := \{\{S, T\} \in  V_{\mathbb{P}}^{(2)}\, : \, \, S \cap T \neq \emptyset \}}$. A set of vertices is stable in $G_{\mathbb{P}}$ if and only if its corresponding collection of subsets is a packing of $[n]$. Hence, the packing family $\bold{F}_n([n])$ is the collection of all stable sets in $G_{\mathbb{P}}$. It follows that there is a bijection between the elements in $\mathcal{P}^n_n$ and the stable set polytope of $G_{\mathbb{P}}$. 

Finally, for $r = 1$, the $r$-packings of $[n]$ are subsets of $[n]$, so the polytope $\mathcal{P}^n_1$ simplifies to
\begin{align} \label{Eq:CliqueDescription}
    \mathcal{R}^n_1 := \left\{ X \in \mathcal{S}^n \, : \, \, X = \sum_{S \subseteq [n]} \theta_S \mathbbm{1}_S \mathbbm{1}_S^\top,~ \sum_{S \subseteq [n]} \theta_S = 1,~\theta_S \geq 0 \text{ for all } S \subseteq [n]  \right\}. 
\end{align}
The polytope $\mathcal{R}^n_1$ relates to the convex hull of the characteristic vectors of all cliques in $K_n$, i.e., the clique polytope of $K_n$. This polytope is in the literature also known as the complete set packing polytope, see~\cite{BulhoesEtAl}. \textcolor{black}{Finally, apart from the embedding, the polytope $\mathcal{R}^n_1$ also coincides with the Boolean quadric polytope~\cite{Padberg}.} 

\medskip

Another continuous formulation of the convex hull of PSD $\{0,1\}$--matrices is given by a conic description. The cone of completely positive matrices is defined as
    $\mathcal{CP}^n := \Conv \left (  \left\{ xx^\top \, : \, \, x \in \mathbb{R}^n_+ \right \} \right )$.
An extension of the completely positive matrices are the so-called set-completely positive matrices, see e.g., \cite{LiederEtAl,BomzeGable}, where the membership condition $x \in \mathbb{R}^n_+$ is replaced by~${x \in \mathcal{K}}$ for a general convex cone $\mathcal{K}$. Lieder et al.~\cite{LiederEtAl} considered the following set-completely positive matrix cone: 
\begin{align} \label{Def:SetCompletelyPositive}
    \mathcal{SCP}^{n} & := \Conv \left( \left\{ xx^\top \, : \, \, x\in \mathbb{R}^n_+,~ x_1 \geq x_i \text{ for all } i \in \{2, \ldots, n\}\right\} \right). 
\intertext{Since the membership condition given in~\eqref{Def:SetCompletelyPositive} is more restricted than $x \in \mathbb{R}^n_+$, we have~${\mathcal{SCP}^n \subsetneq \mathcal{CP}^n}$. Let us now consider the following set-completely positive matrix set:}
    \mathcal{C}^n_1 & := \left\{ X \in \mathcal{S}^n \, : \, \, \begin{pmatrix}
        1 & x^\top \\ x& X
    \end{pmatrix} \in \mathcal{SCP}^{n+1},~~\diag(X) = x\right\}. \label{Eq:SCP_rank1}
\end{align}
The following result follows from Lieder et al.~\cite{LiederEtAl}. 
\begin{theorem}[\cite{LiederEtAl}] \label{Thm:Lieder}
    We have $\mathcal{P}^n_1 = \mathcal{C}^n_1$. 
\end{theorem}

A natural question is whether the descriptions $\mathcal{R}^n_1$ and $\mathcal{C}^n_1$ {\color{black}for~$\mathcal{P}^n_1$} given in~\eqref{Eq:CliqueDescription} and~\eqref{Eq:SCP_rank1}, respectively,  can be extended to higher ranks. The extensions of these sets are as follows:
\begin{align}
    \mathcal{R}^n_r & := \left\{ X \in \mathcal{S}^n \, : \, \, 
        X = \sum_{S \subseteq [n]} \theta_S \mathbbm{1}_S \mathbbm{1}_S^\top,~ \sum_{S \subseteq [n]} \theta_S = r, ~
        \sum_{S: i \in S} \theta_S \leq 1~\forall i \in [n],~\theta_S \geq 0 ~\forall S \subseteq [n]
     \right\}, \\
    \mathcal{C}^n_r & := \left\{ X \in \mathcal{S}^n \, : \begin{pmatrix}
        r & \diag(X)^\top \\ \diag(X) & X
    \end{pmatrix} \in \mathcal{SCP}^{n+1},~~\diag(X) \leq \bold{1}_n \right\}. \label{Eq:SCP_rankr}
\end{align} 
The extension from $\mathcal{C}^n_1$ to $\mathcal{C}^n_r$ follows from the intersection of the Minkowski sum of $r$ copies of~$\mathcal{C}^n_1$ with the upper bound constraint $X \leq \bold{J}_n$. Since $X_{ii} \geq X_{ij}$ for all $i,j \in [n]$ if $X \in \mathcal{C}^n_1$, it suffices to add $\diag(X) \leq \bold{1}_n$. The extension from $\mathcal{R}^n_1$ to $\mathcal{R}^n_r$ is derived as follows. If~$X \in \mathcal{P}^n_r$,  then~$X = \sum_{\mathcal{F} \in \bold{F}_r([n])} \lambda_\mathcal{F} \bold{E}_\mathcal{F}$ for some nonnegative weights $\lambda_{\mathcal{F}}$. By splitting each~$r$-packing into its separate subsets, we obtain
\begin{align*}
    X &= \sum_{\mathcal{F} \in \bold{F}_r([n])} \hspace{-0.1cm} \lambda_\mathcal{F} \bold{E}_\mathcal{F} = \sum_{\mathcal{F} \in \bold{F}_r([n])} \hspace{-0.1cm} \lambda_\mathcal{F} \sum_{S \in \mathcal{F}} \mathbbm{1}_S \mathbbm{1}_S^\top = \sum_{S \subseteq [n]} \sum_{\substack{\mathcal{F} \in \bold{F}_r([n]) : \\ S \in \mathcal{F}}} \lambda_\mathcal{F} \mathbbm{1}_S \mathbbm{1}_S^\top = \sum_{S \subseteq [n]} \theta_S \mathbbm{1}_S \mathbbm{1}_S^\top, 
\end{align*}
where $\theta_S := \sum_{\mathcal{F} \in \bold{F}_r([n]) : S \in \mathcal{F}} \lambda_\mathcal{F}$. Moreover, $\sum_{S \subseteq [n]}\theta_S   = \sum_{\mathcal{F} \in \bold{F}_r([n])} \lambda_{F} |\mathcal{F}| \leq  r \sum_{\mathcal{F} \in \bold{F}_r([n])} \lambda_{F} = r$. By increasing $\theta_{\emptyset}$, we obtain $\sum_{S \subseteq [n]} \theta_S = r$. Finally, since $X_{ii} \leq 1$ for $i \in [n]$, we have~${\sum_{S: i \in S}\theta_S \leq 1}$. We conclude that $\mathcal{P}^n_r \subseteq \mathcal{R}^n_r$. 

Unfortunately, for $r \geq 2$, the sets $\mathcal{R}^n_r$ and $\mathcal{C}^n_r$ do no longer exactly describe $\mathcal{P}^n_r$. 
Namely, consider the matrix $X=\frac{1}{2}P P^\top$ where $P=  \mathbf{I}_3+\mathbf{E}_{23} + \mathbf{E}_{32} + \mathbf{E}_{31} + \mathbf{E}_{41}.$
For this matrix one can verify that $X \in \mathcal{R}^4_2$ and $X \in \mathcal{C}^4_2$, while $X \notin \mathcal{P}^4_2$. For $r \geq 2$,   the following \textcolor{black}{relationship} between  $\mathcal{P}^n_r$, $\mathcal{C}^n_r$, $\mathcal{R}^n_r$ holds. 
\begin{theorem}
    We have $\mathcal{P}^n_r \subseteq {\color{black} \mathcal{C}^n_r = \mathcal{R}^n_r}$, while for $r = 1$ the three sets are equal.
\end{theorem} 
\begin{proof}
   Since $\mathcal{P}^n_r = \Conv(\mathcal{D}_r^n)$, it suffices to consider membership of the elements in $\mathcal{D}^n_r$ in~$\mathcal{C}^n_r$. Let~$X \in \mathcal{D}^n_r$, then $X = \sum_{j=1}^r x_j x_j^\top$ for some $x_j \in \{0,1\}^n$, $j \in [r]$.  Let $Y^j := x_j x_j^\top$ for all~$j \in [r]$. We clearly have $\begin{psmallmatrix}
            1 & \diag(Y^j) \\
            \diag(Y^j) & Y^j
        \end{psmallmatrix} \in \mathcal{SCP}^{n+1} \text{ for all } j \in [r]$,
    from {\color{black}which} it follows that
    \begin{align*}
        \sum_{j = 1}^r \begin{pmatrix}
            1 & \diag(Y^j) \\
            \diag(Y^j) & Y^j
        \end{pmatrix}  = \begin{pmatrix}
            r & \diag(X) \\ \diag(X) & X
        \end{pmatrix}\in \mathcal{SCP}^{n+1}. 
    \end{align*}
    Moreover, $X \in \{0,1\}^{n \times n}$, so $\diag(X) \leq \bold{1}_n$. We conclude that $X \in \mathcal{C}^n_r$.

    To prove {\color{black}$\mathcal{C}^n_r = \mathcal{R}^n_r$}, let $X \in \mathcal{C}^n_r$. We define the matrix $Y$ as
    \begin{align} \label{Def:caligraphicY}
        Y := \frac{1}{r} \begin{pmatrix}
            r & \diag(X)^\top \\ \diag(X) & X
        \end{pmatrix} = \begin{pmatrix}
            1 & \diag(\frac{1}{r}X)^\top \\ \diag(\frac{1}{r}X) & \frac{1}{r}X
        \end{pmatrix}. 
    \end{align}
    From the fact that $X \in \mathcal{C}^n_r$, it follows that $Y \in \mathcal{SCP}^{n+1}$. Applying Theorem~\ref{Thm:Lieder} to the matrix $Y$ implies that $\frac{1}{r} X$ is a convex combination of rank one binary PSD matrices, i.e., there exist $\theta'_S \geq 0$ for all~$S \subseteq [n]$ with $\sum_{S \subseteq [n]}\theta'_S = 1$, such that $\frac{1}{r}X = \sum_{S \subseteq [n]} \theta'_S \mathbbm{1}_S \mathbbm{1}_S^\top$, or equivalently, $X = \sum_{S \subseteq [n]} r\theta'_S \mathbbm{1}_S \mathbbm{1}_S^\top$. 
    Now, let~$\theta_S := r\theta'_S$ for all~$S \subseteq [n]$,  from {\color{black}which it follows that} $\sum_{S \subseteq [n]} \theta_S = r$. Since~${\diag(X) \leq \bold{1}_n}$, it follows that~$X_{ii} = \sum_{S: i \in S}\theta_S \leq 1$. We conclude that $X \in \mathcal{R}^n_r$. 

    {\color{black}Finally, observe that the argument above can also be followed in the converse direction. That is, given $X \in \mathcal{R}^n_r$ with corresponding weights $\theta_S$ for all $S \subseteq [n]$, we define $\theta'_S := \frac{1}{r} \theta_S$, $S \subseteq [n]$, which implies that $\frac{1}{r}X \in \mathcal{P}^n_1$. By Theorem~\ref{Thm:Lieder}, we know that $Y$, see~\eqref{Def:caligraphicY}, is contained in $\mathcal{SCP}^{n+1}$, implying $X \in \mathcal{C}^n_r$.} 
\end{proof}

\subsection[Theory on PSD $\{\pm 1\}$--matrices]{Theory on PSD $\mathbf{\{\pm 1\}}$--matrices}
In this section we  present several results for PSD matrices that have entries in $\{\pm 1\}$. Let us first state the following well-known result.
\begin{proposition}[\cite{AnjosWolkowicz}] Let $X$  be a symmetric matrix. Then, $X \succeq \mathbf{0}$, $X \in \{\pm 1\}^{n \times n}$ if and only if $X = xx^\top$ for some $x \in \{\pm 1\}^n$. \label{Prop:minusonePSD}
\end{proposition}
A simple necessary condition for $X \in \{\pm 1\}^{n \times n}$ to be PSD is that $\diag(X) = \mathbf{1}$. {\color{black} The next result}  establishes  the  equivalence between $\{0,1\}$-- and $\{\pm 1\}$--PSD matrices by exploiting their rank.
\begin{proposition} Let $X \in \{\pm 1\}^{n \times n}$ be {\color{black}a symmetric} matrix and  $Y := \frac{1}{2}(X + \bold{J}) \in \{0,1\}^{n \times n}$. Then, $X \succeq \mathbf{0}$ if and only if  $\diag(Y) = \mathbf{1}$, $Y \succeq \mathbf{0}$ and $\rank(Y) \leq 2$. \label{Prop:binPSDequi}
\end{proposition}
\begin{proof}
$(\Longrightarrow )$: Let $X \succeq \mathbf{0}$. Since $\bold{J} \succeq \mathbf{0}$, it follows that $Y \succeq \mathbf{0}$. Moreover, ${\diag(X) = \diag(\bold{J}) = \mathbf{1}}$ implies that $\diag(Y) = \mathbf{1}$. Finally, by Proposition~\ref{Prop:minusonePSD} we know that~${X = xx^\top}$ for some $x \in \{\pm 1\}^n$. Therefore, $Y$ is the weighted sum of two rank one matrices, so $\rank(Y) \leq 2$. 

$(\Longleftarrow )$: 
Let $Y = \frac{1}{2}(X + \bold{J})\succeq \mathbf{0}$, $\diag(Y) = \mathbf{1}$ and $\rank(Y) \leq 2$ for some symmetric matrix~${X \in \{ \pm 1\}^{n \times n}}$.
Since $Y$ is binary positive semidefinite with rank at most two, it follows from Theorem~\ref{Thm:binPSD} that ${Y = x_1x_1^\top + x_2x_2^\top}$ for some $x_1, x_2 \in \{0,1\}^n$. Then,
$X  = 2Y - \bold{J} =  (x_1 - x_2)(x_1 - x_2)^\top,$
which implies that $X \succeq \mathbf{0}$. 
\end{proof} 
Note that the matrix $Y$ from the previous theorem has rank one if and only if $Y=X=\bold{J}.$
Similar to~\eqref{Def:setD}, we define the discrete set of all $\{\pm 1\}$--matrices as 
\begin{align}
    \widehat{\mathcal{D}}^n := \left\{X \in \{\pm 1\}^{n \times n} \, : \, \, X \succeq \bold{0} \right\}, \label{Def:setDplusminus}
\end{align}
where the subscript $r$ is not present anymore, as all matrices in $\widehat{\mathcal{D}}^n$ have rank one. Based on Proposition~\ref{Prop:minusonePSD}, we can easily establish that  $|\widehat{\mathcal{D}}^n| = 2^{n-1}$. 
Next, we summarize  known results on sets related to $\{\pm 1\}$--matrices. The convex hull of all PSD $\{\pm1\}$--matrices is known as the cut polytope:
\begin{align} \label{Def:setPplusminus}
    \widehat{\mathcal{P}}^n := \Conv(\widehat{\mathcal{D}}^n),
\end{align}
see e.g.,~\cite{LaurentPoljakRendl}.
Also, we define the following set-completely positive matrix cone:
\begin{align} \label{Def:SetCompletelyPositivePM}
    \mathcal{SC}\widehat{\mathcal{P}}^{n} & := \Conv \left( \left\{ xx^\top \, : \, \, x\in \mathbb{R}^n,~ x_1 + x_i \geq 0, \, x_1-x_i \geq 0 \text{ for all } i \in \{2, \ldots, n\}\right\} \right). 
\end{align}    
The cone $\mathcal{SC}\widehat{\mathcal{P}}^{n}$ is considered in~\cite{LiederEtAl}, where the authors show that  $\mathcal{SC}\widehat{\mathcal{P}}^{n}$ and $\mathcal{SCP}^{n}$, see~\eqref{Def:SetCompletelyPositive}, are related as follows: ${\mathcal T}( \mathcal{SCP}^{n} ) = \mathcal{SC}\widehat{\mathcal{P}}^{n}$
 and
  ${\mathcal T}^{-1}(\mathcal{SC}\widehat{\mathcal{P}}^{n}) = \mathcal{SCP}^{n}$,
where ${\mathcal T}$  is an appropriate linear mapping. 
Lieder et al.~\cite{LiederEtAl}  consider the following set-completely positive matrix set:
\begin{align}
    \widehat{\mathcal{C}}^n & := \left\{ X \in \mathcal{S}^n \, : \, \, \begin{pmatrix}
        1 & x^\top \\ x& X
    \end{pmatrix} \in \mathcal{SC}\widehat{\mathcal{P}}^{n+1},~\diag(X) = \mathbf{1}_n \right\}, \label{Eq:SCP_rank1pm1}
\end{align}
which is the analogue of the set $\mathcal{C}^n_1$ for $\{0,1\}$--matrices, see~\eqref{Eq:SCP_rank1}.
\begin{theorem}[\cite{LiederEtAl}] \label{Thm:Lieder2}
    We have $\widehat{\mathcal{P}}^n = \widehat{\mathcal{C}}^n$. 
\end{theorem}
This theorem is the analogue of~\Cref{Thm:Lieder} that provides a result for $\{0,1\}$--matrices. 
For the equivalence transformation between $\{\pm 1\}$-- and $\{0,1\}$--representations of SDP relaxations of binary quadratic optimization problems,
 we refer the interested reader to Helmberg~\cite{Helmberg1997FixingVI}. 

\subsection[Theory on PSD $\{0,\pm 1\}$--matrices]{Theory on PSD $\mathbf{\{0,\pm 1\}}$--matrices}
In the sequel we generalize several results from 
the previous sections to PSD $\{0,\pm 1\}$--matrices. The following result shows that a PSD $\{0,\pm 1\}$--matrix is block-diagonalizable, which is the analogue of \Cref{Prop:binaryProperties}  for~$\{0,1\}$--matrices.
\begin{proposition}\label{prop:ternary1}
Let $X\in \{0, \pm 1 \}^{n \times n}$ be symmetric.
Then, the following statements are equivalent:
\begin{itemize}
\item[(i)]  $\diag({X})=\mathbf{1}_n$, $\rank({X})=r$, ${X} \succeq \mathbf{0}$.
\item[(ii)] There exists a permutation matrix ${Q}$ such that ${QXQ}^\top=\bold{B}_{n_1} \oplus \, \cdots \, \oplus \bold{B}_{n_r} $,  where~$\bold{B}_{n_i} = b_ib_i^\top$, $b_i \in \{\pm1\}^{n_i}$ for $i\in [r]$,
$n = n_1 + \dots + n_r$.
\end{itemize}
\end{proposition}
\begin{proof}
Suppose that ${QXQ}^\top$ is in the block form given in $(ii)$, then it trivially satisfies the conditions given in  $(i)$.
Conversely, let $X \in \{0, \pm 1 \}^{n \times n}$ satisfy $(i)$. Let us consider the $i$th row in $X$. Suppose~$j$ and $k$ are two distinct indices not equal to $i$ in the support of this row, i.e., $X_{ij}, X_{ik} \neq 0$. For the sake of contradiction, suppose that $X_{jk} = 0$. Then, the submatrix of $X$ {\color{black}induced} by $i$, $j$ and $k$ is either one of the following matrices: 
\begin{align*}
    \bordermatrix{
     & i & j & k \cr
i & 1 & 1 & 1 \cr
j & 1 & 1 & 0 \cr
k & 1 & 0 & 1
}, \quad \bordermatrix{
     & i & j & k \cr
i & 1 & -1 & -1 \cr
j & -1 & 1 & 0 \cr
k & -1 & 0 & 1
},  \quad \bordermatrix{
     & i & j & k \cr
i & 1 & 1 & -1 \cr
j & 1 & 1 & 0 \cr
k & -1 & 0 & 1
}, \quad \text{or} \quad \bordermatrix{
     & i & j & k \cr
i & 1 & -1 & 1 \cr
j & -1 & 1 & 0 \cr
k & 1 & 0 & 1
}. 
\end{align*}
One easily checks that the determinants of these matrices are all negative, contradicting that~$X \succeq \bold{0}$. Hence,~$X_{jk} \neq 0$. This argument can be repeated to conclude that the submatrix of~$X$ indexed by the support of row $i$ has entries in $\{\pm 1\}$. Since {\color{black}the submatrix of~$X$} is also positive semidefinite, it follows from Proposition~\ref{Prop:minusonePSD} that the submatrix is of the form $bb^\top$ with $b \in \{\pm 1\}^{n_i}$ for some positive integer $n_i$. 

By the same argument, it follows that the other indices in the submatrix induced by row~$i$ have the same support as row $i$. Indeed, if this would not be the case, one of the four matrices above should be a submatrix of $X$. We conclude that $X$ can be fully constructed from nonoverlapping submatrices of the form $bb^\top$ with $b \in \{\pm 1\}^{n_i}$ for some positive integer~$n_i$. Since its rank equals $r$, there must be $r$ of those submatrices. From here the claim follows.
\end{proof}

\Cref{prop:ternary1} extends easily to the following result.
\begin{corollary}\label{cor:pm1zero}
Let $X\in \{0, \pm 1 \}^{n \times n}$ be symmetric.
Then, the following statements are equivalent:
\begin{itemize}
\item[(i)]   $\rank({X})=r$, ${X} \succeq \mathbf{0}$.
\item[(ii)] There exists a permutation matrix ${Q}$ such that
${QXQ}^\top=\bold{B}_{n_1} \oplus \, \cdots \, \oplus \bold{B}_{n_r}\oplus {\mathbf 0}_{n_{z}\times n_z} $  where $\bold{B}_{n_i} = b_ib_i^\top$, $b_i\in \{\pm1\}^{n_i}$ for $i\in [r]$,
$n = n_1 + \dots + n_r + n_{z}$.
\end{itemize}
\end{corollary}
\begin{proof}
{\color{black}The proof is similar to the proof of~\Cref{Cor:binaryProperties}.}
\end{proof}
Let $X \in \{0, \pm 1 \}^{n \times n}$ be given as in Corollary~\ref{cor:pm1zero}, then
\begin{align*}
QXQ^\top  = \bold{B}_{n_1} \oplus \, \cdots \, \oplus \bold{B}_{n_r}\oplus {\mathbf 0}_{n_{z}\times n_z}  
= b_1b_1^\top  \oplus \, \cdots \, \oplus   b_rb_r^\top \oplus {\mathbf 0}_{n_{z}\times n_z} 
= \sum_{i=1}^r \bar{x}_i \bar{x}_i^\top,
\end{align*}
where $\bar{x}^\top_1 = [b_1^\top~ {\mathbf 0}_{n-n_1}^\top]$,    $\bar{x}^\top_2 = [{\mathbf 0}_{n_{1}}^\top~ b_2^\top~ {\mathbf 0}_{n-n_1-n_2}^\top]$, and so on.
Let $x_i := Q \bar{x}_i$ for  $i\in [r]$, then 
$X = \sum_{i=1}^r {x}_i {x}_i^\top$,
where $x_i \in \{0,\pm 1 \}^n.$ 
This construction yields the following decomposition of PSD $\{0,\pm 1\}$--matrices.
\begin{theorem}
Let ${X} \in \{0,\pm 1\}^{n \times n}$ be symmetric. 
Then, ${X} \succeq \mathbf{0}$ if and only if~${X} = \sum_{j = 1}^r x_j x_j^\top$ for some~${x_j} \in \{0,\pm 1\}^n$, $j \in [r]$. \label{prop:pm1PSD}
\end{theorem}
The previous result is 
an extension of~\Cref{Thm:binPSD} to $\{0,\pm 1\}$--matrices.
We now consider an equivalence between a PSD $\{0,\pm 1\}$--matrix of rank one and an extended linear matrix inequality, i.e., the analogue of Theorem~\ref{Thm:binPSDrank}.
\begin{proposition}
\label{prop:pm1b}
Let $Y = \begin{pmatrix} 1 &  x^\top  \\ {x} & {X}\end{pmatrix} \in {\mathcal S}^{n+1}$ {\color{black} with} $\textup{supp}(\diag(X))= \textup{supp}(x)$. Then,
 $Y \in \{0,\pm 1\}^{(n+1)\times (n+1)}$, $Y \succeq {\mathbf 0}$ if and only if $X=xx^\top.$
\end{proposition}
\begin{proof}
Let $Y_{ij} \in \{0, \pm 1\}$ for all $i,j \in [n+1]$ and $Y\succeq \mathbf{0}$. Then $x\in \{0, \pm 1\}^{n}$. The Schur complement lemma implies ${X} - {x}{x}^\top \succeq \mathbf{0}$.
If ${X}_{ii}=0$ then ${x}_i=0$, and if ${X}_{ii}=1$ then ${x}_i=1$ or~${x}_i=-1$.
Thus $\diag({X}- {x}{x}^\top)=\mathbf{0}$,
from {\color{black}which} it follows that $X= xx^\top$.
The converse direction is trivial.
\end{proof}
Clearly, the condition  $\mbox{supp}(\diag(X))= \mbox{supp}(x)$ can be replaced by $\diag(X)_{ii}=|x_i|$ for all~$i\in [n]$, where $|\cdot |$ denotes the absolute value.
\section{Binary quadratic optimization problems}
\label{Sect:BQPformulations}

In this section we exploit the theoretical results on discrete PSD matrices from the previous section to derive exact reformulations of binary quadratic programs (BQPs) as binary semidefinite programs. 
In Section~\ref{section:vector lifting} we consider the general class of binary quadratically constrained quadratic programs. In Section~\ref{section:matrix lifting} we consider a subclass of these programs that allow for a formulation as a binary quadratic matrix program.

\subsection{Binary quadratically constrained quadratic programs}
\label{section:vector lifting}

A quadratically constrained quadratic program (QCQP) is an optimization problem with a quadratic objective function under the presence of quadratic constraints. Many discrete optimization problems can be formulated as QCQPs. 

Let $Q_0, Q_i \in \mathcal{S}^n$, $c_0, c_i  \in \mathbb{R}^n$ for all $i \in [m]$, and ${a_i} \in \mathbb{R}^n$, $b_i \in \mathbb{R}$ for all $i \in [p]$, where~$m,p\in \mathbb{N}$.
We consider binary programs of the following  form: 
\begin{align} 
& \begin{aligned}
\min \quad & {x}^\top Q_0 {x} + c_0^\top {x} \\
\text{s.t.} \quad  & x^\top Q_i x + c_i^\top x \leq  d_i \quad \forall  i\in [m] \\
& {a_i}^\top {x} = b_i \quad \forall i \in [p] \\
& {x} \in \{0,1\}^n. \end{aligned} \label{BQP01} \tag{$QCQP$} \\
\intertext{The quadratic terms in \eqref{BQP01} can be written as $\langle {Q_i}, {X} \rangle + c_i^\top {x}$ for all $i$, where we substitute~${X}$ for~${x}{x}^\top.$  
This yields the following exact {\color{black}reformulation}  of~\eqref{BQP01}: }
& \begin{aligned}
\min \quad & \langle Q_0 , X \rangle  + c_0^\top {x} \\
\text{s.t.} \quad  &  \langle Q_i, X  \rangle + c_i^\top x \leq d_i \quad \forall  i\in [m] \\
& {a_i}^\top {x} = b_i \quad \forall i \in [p] \\
& Y = \begin{pmatrix}
1 & {x}^\top \\ {x} & {X}
\end{pmatrix} \succeq \bold{0},~\diag({X}) = {x},~ \mbox{rank}(Y)=1. \end{aligned}\nonumber \\
\intertext{Here we used the conventional notion of exactness, i.e., the nonconvex constraint~${\rank(Y) = 1}$. {\color{black} We also
exploit here Theorem~\ref{Thm:binPSDrank} in order to not explicitly require $x$ to be binary.}
However, one can utilize an alternative notion of exactness in terms of integrality, namely by exploiting Theorem~\ref{Thm:binPSDrank}. This leads to the following binary semidefinite program (BSDP):}
& \begin{aligned}  
\min \quad & \langle Q_0, {X} \rangle + c_0^\top {x} \\
\text{s.t.} \quad  &  \langle Q_i, X  \rangle + c_i^\top x \leq d_i \quad \forall  i\in [m] \\
& {a_i}^\top {x} = b_i \quad \forall i \in [p] \\
& \begin{pmatrix}
1 & {x}^\top \\ {x} & {X}
\end{pmatrix} \succeq \bold{0},~\diag({X}) = {x} , ~{\color{black}x \in \{0,1\}^{n}}.
\end{aligned} \label{BQP01_bin} \tag{$BSDP_{QCQP}$}
\end{align}
\textcolor{black}{Observe that it is sufficient to impose integrality on the diagonal of $X$. Namely, it follows from the determinants of the $3 \times 3$ principal submatrices of the matrix $Y$ that $X_{ij} \in \{0,1\}$ whenever $X_{ii},~X_{jj} \in \{0,1\}$ for all $i$ and $j$, see e.g., \cite[Section 3.2]{HelmbergThesis}.}
 Note that a binary matrix $X$ that satisfies the linear matrix inequality from \eqref{BQP01_bin} with~$x=\diag(X)$, is an element of  $\mathcal{D}^n_1$, see \eqref{Eq:SetDdiscreteLMI}. 
The next result follows directly from the previous discussion and Theorem~\ref{Thm:binPSDrank}.
\begin{theorem}  \label{Thm:BQP01_Equivalence} 
\eqref{BQP01_bin} is equivalent to~\eqref{BQP01}. 
\end{theorem}

To provide a more compact BSDP formulation of~\eqref{BQP01}, we prove the following result.
\begin{lemma} \label{Lemma:CompactBSDP}
Let~$S= {\sum_{i=1}^p}
\begin{psmallmatrix} {-b_i} \\ {a_i} \end{psmallmatrix}
\begin{psmallmatrix} {-b_i} \\ {a_i} \end{psmallmatrix}^\top
$ and $Y=\begin{psmallmatrix}
1 & {x}^\top  \\ {x} & {X}
\end{psmallmatrix} \succeq \mathbf{0}$, where $\diag({X}) = {x}$ and ${X \in \{0,1\}^{n \times n}}$. 
Then, $a_i^\top x=b_i$ for all $i\in [p]$ if and only if $\langle S, Y \rangle =0$.
\end{lemma}
\begin{proof}
It follows from \Cref{Thm:binPSDrank} that
$Y=\begin{psmallmatrix} 1 \\ x \end{psmallmatrix}\begin{psmallmatrix} 1 \\ x \end{psmallmatrix}^\top$.
If $a_i^\top x=b_i$ for all~$i\in [p]$, it is not difficult to verify that $\langle S, Y \rangle =0$.
Conversely, let $\langle S, Y \rangle =0$. Then,
$
0 =  \sum_{i=1}^p \left \langle
\begin{psmallmatrix} {-b_i} \\ {a_i} \end{psmallmatrix}
\begin{psmallmatrix} {-b_i} \\ {a_i} \end{psmallmatrix}^\top, \begin{psmallmatrix} 1 \\ x \end{psmallmatrix}\begin{psmallmatrix} 1 \\ x \end{psmallmatrix}^\top \right \rangle 
=
 \sum_{i=1}^p (b_i-a_i^\top x)^2,
$
from {\color{black}which} it follows that $a_i^\top x=b_i$ for all  $i\in [p]$.
\end{proof}
Lemma~\ref{Lemma:CompactBSDP} induces the following compact BSDP that is equivalent to~\eqref{BQP01}:
\begin{align*}  
\begin{aligned}
\min \quad & \langle Q_0, {X} \rangle + c^\top {x} \\
\text{s.t.} \quad  &  \langle Q_i, X  \rangle + c_i^\top x \leq  d_i \quad \forall  i\in [m] \\
&  \sum\limits_{i=1}^p  \left \langle 
\begin{pmatrix} {-b_i} \\ {a_i} \end{pmatrix}
\begin{pmatrix} {-b_i} \\ {a_i} \end{pmatrix}^\top,\begin{pmatrix}
1 & {x}^\top \\ {x} & {X}
\end{pmatrix}  \right \rangle =0\\
& \begin{pmatrix}
1 & {x}^\top \\ {x} & {X}
\end{pmatrix} 
\succeq \bold{0},~\diag({X}) = {x},~{\color{black}x \in \{0,1\}^{n}}.  
\end{aligned}
\end{align*}
There are various equivalent formulations of the  binary quadratic program~\eqref{BQP01} in the literature. We finalize this subsection by mentioning below only those that are closely related to our approach.

Assume that  $Q_i=\mathbf{0}$, $c_i=\mathbf{0}$, and $d_i=0$ for all $i\in[m]$ in~\eqref{BQP01}.
Burer~\cite{Burer2009} proved that the resulting  optimization problem with  a quadratic objective and linear constraints is equivalent to the following completely positive program:
\begin{align*}  
& \begin{aligned}
\min \quad & \langle Q_0, {X} \rangle + c^\top {x} \\
\text{s.t.} \quad  &  {a_i}^\top {x} = b_i \quad \forall i \in [p], ~\langle {a_i}{a_i}^\top , {X} \rangle = b_i^2 \quad \forall i \in [p] \\
& \begin{pmatrix}
1 & {x}^\top \\ {x} & {X}
\end{pmatrix}  \in \mathcal{CP}^{n+1},~  \diag({X}) = {x} ,
\end{aligned}
\intertext{
provided that the inequalities $0\leq x_i \leq 1$ for $i\in [n]$ are implied by the constraints of the original problem.
Here $\mathcal{CP}^{n+1}$ is the cone of completely positive matrices. \endgraf
On the other hand, Lieder et al.~\cite{LiederEtAl} proved the following  equivalent formulation of the BQP {\color{black} with quadratic objective and linear constraints}:}
& \begin{aligned}
\min \quad & \langle Q_0, {X} \rangle + c^\top {x} \\
\text{s.t.} \quad  &  {a_i}^\top {x} = b_i \quad \forall i \in [p] \\
& \begin{pmatrix}
1 & {x}^\top \\ {x} & {X}
\end{pmatrix}  \in \mathcal{SCP}^{n+1},~ \diag({X}) = {x}, 
\end{aligned}
\end{align*}
where the cone  $\mathcal{SCP}^{n+1}$ is defined in~\eqref{Def:SetCompletelyPositive}.
The authors of~\cite{LiederEtAl} also proved that, under mild assumptions, the binary quadratic problem~\eqref{BQP01} with also quadratic constraints  
can be equivalently reformulated as an optimization problem over the set-completely positive matrix cone $\mathcal{SCP}^{n+1}$. 

\medskip 

We end this section by providing an example of a problem that can be modeled as~\eqref{BQP01_bin}.

\begin{example}[The stable set problem]\label{example:stable set1} 
Let $G = (V,E)$ be a simple graph on $n$ vertices. A stable set in $G$ is a subset $S \subseteq V$ such that no two vertices in $S$ are adjacent in $G$. The stable set problem (SSP) asks for the largest size of a stable set in $G$. To model this problem, let~$x \in \{0,1\}^ n$ be such that $x_i = 1$ if $i \in S$ and~$x_i = 0$ otherwise. Then, $x$ is the 
characteristic  vector  of a stable set  in $G$ if $x^\top (\mathbf{E}_{ij} + \mathbf{E}_{ij}^\top) x = 0$ for all~$\{i,j\} \in E$. The cardinality of the stable set equals $x^\top x$, hence the SSP is of the form~\eqref{BQP01}. Applying Theorem~\ref{Thm:BQP01_Equivalence}, the following BSDP models the SSP: 
\begin{align} \label{Eq:stabilitynum}
    \begin{aligned}
    \alpha(G):=   \max \quad & \langle \bold{I}_n,~ X \rangle \\
        \text{s.t.} \quad & X_{ij} = 0 \quad \forall \{i,j\} \in E \\
        & \begin{pmatrix}
            1 & x^\top \\ x & X
        \end{pmatrix} \succeq \bold{0},~\diag(X) = x,~{\color{black}x \in \{0,1\}^{n}}. 
    \end{aligned}
\end{align} 
The doubly nonnegative relaxation of~\eqref{Eq:stabilitynum} obtained after replacing $X \in \{0,1\}^{n \times n}$ by~${\bold{0} \leq X \leq \bold{J}}$, is well-studied in the literature, see e.g.,~\cite{GrotschelEtAl}. It is equivalent to a strengthened version of the Lov\'asz theta number, known as the Schrijver $\vartheta'$-number~\cite{SchrijverTheta}. \hfill \qed
\end{example}

\subsection{Binary quadratic matrix programs}
\label{section:matrix lifting}

A quadratic matrix program (QMP) \cite{Beck} is a programming formulation where {\color{black}the objective and constraints are given by}
\begin{align} \label{Eq:QMPconstraint}
   \tr(P^\top Q_i P) + 2\tr(B_i^\top P) + d_i
\end{align}
for some $Q_i \in \mathcal{S}^n$, $B_i \in \mathbb{R}^{n \times k}$ and $c_i \in \mathbb{R}$, where $P$ is an $n\times k$ matrix variable. QMPs are a special case of QCQPs and are particularly useful to model optimization problems where the matrix $P$ has entries in $\{0,1\}$ and represents a classification of $n$ objects over $k$ classes, i.e., $P_{ij} = 1$ if and only if object $i$ is assigned to class $j$. For example, if each object needs to be assigned in exactly (resp.~at most) one class, we call $P$ a partition (resp.~packing) matrix. 

In this section we consider two different binary QMPs of increasing generality and show how these can be reformulated as BSDPs.  For both QMPs, we consider some well-known problems that fit {\color{black}into} the framework. 

Our first QMP incorporates a specific objective and constraint structure, while optimizing over the packing or partition matrices.
Let $Q_0, Q_i \in \mathcal{S}^n$, $d_i \in \mathbb{R}$ for all $i \in [m]$, $a_i \in \mathbb{R}^n$ and~$b_i \in \mathbb{R}_+$ for all $i \in [p]$.  We consider the binary quadratic matrix program
\begin{align} \tag{$QMP_1$} \label{QMP_1} 
    \begin{aligned}
        \min \quad & \tr({P}^\top {Q}_0 P) \\
        \text{s.t.} \quad & \tr(P^\top Q_i P) + d_i \leq 0 \quad \forall i \in [m],~ P^\top a_i \leq b_i \bold{1}_k \quad \forall i \in [p] \\
        & P \bold{1}_k \leq \bold{1}_n,~P \in \{0,1\}^{n \times k}. 
    \end{aligned} 
\end{align}
Observe that $P\bold{1}_k \leq \bold{1}_n$ implies that $P$ is a packing matrix. This constraint is replaced by~$P\bold{1}_k = \bold{1}_n$ in case we deal with partition matrices. The constraints~${\tr(P^\top Q_i P) + d_i \leq 0}$ and~$P^\top a_i \leq b_i \bold{1}_k$ might follow from the structure of the problem under consideration. Observe that these constraints differ from the general form~\eqref{Eq:QMPconstraint} in the sense that the linear part $\tr(B_i^\top P)$ is only included in a very specific form. 

A possible way to deal with the quadratic terms in~\eqref{QMP_1} is by lifting the variables in a higher-dimensional space.  By vectorizing the matrix $P$, the problem~\eqref{QMP_1} can be written in the form~\eqref{BQP01}, after which we can follow the approach of Section~\ref{section:vector lifting}. This results in a BSDP where the matrix variable is of order $nk + 1$. 
Since the resulting program is obtained from a lifting of the vectorization of $P$,  we say that we applied a vector-lifting approach. To obtain a more compact problem formulation where the matrix variable is of lower order, we here consider a matrix-lifting approach. In particular, the objective function can be written as $\tr( P^\top Q_0 P) = \tr(Q_0 PP^\top ) = \tr(Q_0 X)$,
where $X = PP^\top$. By doing so, we obtain the following BSDP:
\begin{align} \tag{$BSDP_{QMP1}$} \label{ISDP_ML1}
   \begin{aligned}
        \min \quad & \langle {Q}_0, X \rangle \\
        \text{s.t.} \quad & \langle {Q}_i, X \rangle + d_i \leq 0 \quad \forall i \in [m],~ X a_i \leq b_i x \quad \forall i \in [p] \\
        & \begin{pmatrix}
            k & x^\top \\
            x & X
        \end{pmatrix} \succeq \bold{0},~\diag(X) = x,~X \in \{0,1\}^{n \times n}. 
    \end{aligned} 
\end{align}
If a QMP is defined over the partition matrices, then  $P\bold{1}_k \leq  \bold{1}_n$ is replaced by~${P\bold{1}_k = \bold{1}_n}$ in~\eqref{QMP_1}, and consequently 
 $\diag(X) = x$ is  replaced
by $\diag(X) = \bold{1}_n$ in~\eqref{ISDP_ML1}.
By exploiting theory from Section~\ref{Subsect:theory01}, we  show the following equivalence.

\begin{theorem} \label{THM:QMP_1_Equivalence}
    \eqref{ISDP_ML1} is equivalent to \eqref{QMP_1}. 
\end{theorem}
\begin{proof} Let $P$ be feasible for~\eqref{QMP_1} and define $X = PP^\top$ and $x = P\bold{1}_k$. Since $P$ represents a packing matrix, we have $X \in \{0,1\}^{n \times n}$, where $x$ is a $\{0,1\}$--vector indicating whether object $i$ is packed in one of the classes or not. Then, ${\langle Q_i, X \rangle + d_i = \langle Q_i, PP^\top \rangle + d_i = \tr(P^\top Q_i P) + d_i \leq 0}$ for all~$i \in [m]$. Moreover, we have $X a_i = PP^\top a_i \leq b_i P \bold{1}_k = b_i x$. To show that $\diag(X) = x$, observe that~${X_{ii} = \sum_{j =1}^k P_{ij}^2 = \sum_{j = 1}^k P_{ij} = \mathbf{e}_i^\top P\bold{1}_k = x_i}$. 
Finally, we can decompose the matrix~$\begin{psmallmatrix}
            k & x^\top \\ x & X \end{psmallmatrix}$ into~$\begin{psmallmatrix}
            k & x^\top \\
            x & X
        \end{psmallmatrix} = \begin{psmallmatrix}
            \bold{1}_k^\top \\ P
        \end{psmallmatrix} \begin{psmallmatrix}
            \bold{1}_k^\top \\ P
        \end{psmallmatrix}^\top,$
showing that it is PSD. We conclude that $X$ and $x$ are feasible for~\eqref{ISDP_ML1}.  

To show the converse inclusion, let $X$ and $x = \diag(X)$ be feasible for~\eqref{ISDP_ML1}. It follows from Corollary~\ref{Cor:binPSDrank_upper} that $X$ can be decomposed as the sum of at most $k$ rank-one symmetric $\{0,1\}$--matrices. By adding copies of the zero matrix in case $\rank(X) < k$, we may assume that there exist~$x_1, \ldots, x_k \in \{0,1\}^n$ such that $X = \sum_{j=1}^k x_jx_j^\top$.
Now, let $P = [
    x_1~\dots~x_k ]$. Then, $P \in \{0,1\}^{n \times k}$ with $P\bold{1}_k = \sum_{j=1}^k x_j = \diag(X) \leq \bold{1}_n$. To prove that $P^\top a_i \leq b_i\bold{1}_k$, consider column $j^*$ of $P$. {\color{black}If} all entries in $P\bold{e}_{j^*} ~(= x_{j^*})$ are zero, implying that $\bold{e}_{j^*}^\top P^\top a_i = 0 \leq b_i$, since $b_i \in \mathbb{R}_+$.
Otherwise, there exists a row $i^*$ such that $P_{i^*j^*} = 1$. For the~$i^*$th row of $X$, we know $\bold{e}_{i^*}^\top X =\sum_{j=1}^k(x_j)_{i^*} x_j^\top = x_{j^*}^\top$.
The $i^*$th row of the system $Xa_i \leq b_i x$ then reads $x_{j^*}^\top a_i \leq b_i x_{i^*} = b_i$. Hence,~${P^\top a_i \leq b_i \bold{1}_k}$. Finally, the constraint $\tr(P^\top Q_i P) + d_i \leq 0$ follows immediately from $\langle Q_i, X \rangle +d_i \leq 0$ for all $i \in [m]$. Thus,~$P$ is feasible for~\eqref{QMP_1}. 
    
    As the objective functions of~\eqref{QMP_1} and \eqref{ISDP_ML1} clearly coincide with respect to the given mapping between $P$ and $X$, we conclude that the two programs are equivalent.
\end{proof}
The matrix $P$ {\color{black}no longer appears} explicitly in~\eqref{ISDP_ML1}, and therefore we will not be able to write all quadratic problems over the packing or partition matrices in this form. The typical problems that can be modeled as~\eqref{ISDP_ML1}, are the ones that are symmetric over the classes~$[k]$, i.e., we do not add constraints for one specific class. 
Below we discuss two examples from the literature that fit in the framework of~\eqref{QMP_1}.

\begin{example}[The maximum $k$-colorable subgraph problem]
    Let $G = (V,E)$ be a simple graph with~$n := |V|$ and $m := |E|$. Given a positive integer $k$, a graph is called $k$-colorable if it is possible to assign to each vertex in $V$ a color in $[k]$ such that any two adjacent vertices get assigned a different color. The maximum~$k$-colorable subgraph  (M$k$CS) problem, see e.g., \cite{KuryatnikovaEtAl, Narasimhan}, asks to find an induced subgraph $G' = (V', E')$ of~$G$ that is $k$-colorable such that $|V'|$ is maximized.

    The M$k$CS problem can be modeled as~\eqref{QMP_1} where $P \in \{0,1\}^{n \times k}$ is such that $P_{ij} = 1$ if and only if vertex $i \in [V]$ is in color class $j \in [k]$. 
    In order to model that $P$ induces a coloring in~$G$, we include the constraints $\tr(P^\top (\mathbf{E}_{ij} + \mathbf{E}_{ji}) P) = 0$ for all $\{i,j\} \in E$.
    Additional constraints of the form~$P^\top a_i\leq b_i \bold{1}_k$ do not appear in the formulation. 

    Now, it follows from Theorem~\ref{THM:QMP_1_Equivalence} that the M$k$CS problem can be modeled as the following BSDP: 
\begin{align}  \label{BSDP_{MkCS}}
 \begin{aligned}
        \max \quad & \langle \bold{I}_n, X \rangle \\
        \text{s.t.} \quad & X_{ij} = 0 \quad \forall \{i,j\} \in E,~ \begin{pmatrix}
            k & x^\top \\
            x & X
        \end{pmatrix} \succeq \bold{0},~ \diag(X) = x,~X \in \{0,1\}^{n \times n}.
    \end{aligned} 
\end{align}
The doubly nonnegative relaxation of~\eqref{BSDP_{MkCS}} obtained after replacing~${X \in \{0,1\}^{n \times n}}$ by~${\bold{0} \leq X \leq \bold{J}}$, equals the formulation $\theta^3_k(G)$ derived in~\cite{KuryatnikovaEtAl}.  \hfill \qed
\end{example}

The next example shows that the parameter $k$ in~\eqref{ISDP_ML1} can also be used as a variable in order to quantify the number of classes in the solution. 

\begin{example}[The quadratic bin packing problem]
    Let a set of $n$ items  be given, each with a positive weight~$w_i \in \mathbb{R}_+$, $i \in [n]$. Assume an unbounded  number of bins is available, each with total capacity~$W \in \mathbb{R}_+$ and cost $c \in \mathbb{R}_+$.  Moreover, let $D = (d_{ij}) \in \mathcal{S}^n$ denote a dissimilarity matrix, where $d_{ij} $ equals the cost of packing item $i$ and $j$ in the same bin. The goal of the quadratic bin packing problem (QBPP), see e.g., \cite{Chagas}, is to assign each item to exactly one bin, under the condition that the total sum of weights for each bin does not exceed $W$, such that the sum of the total dissimilarity and the cost of the used bins is minimized. 

    Let us first consider the related problem where the number of available bins equals $k$. This problem can be modeled in the form~\eqref{QMP_1}, where $P \in \{0,1\}^{n \times k}$ is a matrix with~$P_{ij} = 1$ if and only if item $i$ is contained in bin $j$. Since all items need to be packed, we require $P$ to be a partition matrix, i.e., $P\bold{1}_k = \bold{1}_n$. Moreover, the capacity constraints can be modeled as $P^\top w \leq W \bold{1}_k$. Theorem~\ref{THM:QMP_1_Equivalence} shows that this problem can be modeled as a BSDP where the number of bins $k$ appears as a parameter. If we replace $k$ by a variable~$z$, we obtain the following formulation of the QBPP: 
\begin{align}  \label{BSDP_{QBPP}}
     \begin{aligned}
        \min \quad & \left \langle \begin{pmatrix}
            z & \bold{1}_n^\top \\
            \bold{1}_n & X
        \end{pmatrix},~c \oplus D  \right \rangle \\
        \text{s.t.} \quad & Xw \leq W \bold{1}_n,~\diag(X) = \bold{1}_n,~ \begin{pmatrix}
            z & \bold{1}_n^\top \\
            \bold{1}_n & X
        \end{pmatrix} \succeq \bold{0},~X \in \{0,1\}^{n \times n},~z \in \mathbb{R}.
    \end{aligned} 
\end{align}
The variable $z$ is not explicitly restricted to be integer, since at an {\color{black}optimal} solution it will always be equal to $\rank(X)$. 
\hfill \qed
\end{example}
The quadratic matrix program~\eqref{QMP_1} only includes specific types of constraints of the form~\eqref{Eq:QMPconstraint}. 
We now consider a generalization of~\eqref{QMP_1}. Let $Q_0, Q_i \in \mathcal{S}^n$, $B_0, B_i \in \mathbb{R}^{n \times k}$ and~$d_0, d_i \in \mathbb{R}$ for all~$i \in [m]$ and consider the quadratic matrix program
\begin{align} \tag{$QMP_2$} \label{QMP_2} 
   &  \begin{aligned}
        \min \quad & \tr({P}^\top {Q}_0 P) + 2\tr(B_0^\top P) + d_0 \\
        \text{s.t.} \quad & \tr(P^\top Q_i P) + 2\tr(B_i^\top P) + d_i \leq 0 \quad \forall i \in [m] \\
        & P \bold{1}_k \leq \bold{1}_n,~P \in \{0,1\}^{n \times k}. 
    \end{aligned} 
\intertext{Again, the constraint $P\bold{1}_k \leq \bold{1}_n$ can be replaced by $P \bold{1}_k = \bold{1}_n$ when  optimizing over partition matrices. Now, let us consider the BSDP}
 \tag{$BSDP_{QMP2}$} \label{ISDP_ML2}
& \begin{aligned}
        \min \quad & \left \langle \begin{pmatrix}
\frac{d_0}{k}\bold{I}_k & B_0^\top \\
B_0^\top & {Q}_0
\end{pmatrix}, \begin{pmatrix}
\bold{I}_k & {P}^\top \\
{P} & {X}
\end{pmatrix} \right \rangle \\
\text{s.t.}\quad & \left \langle \begin{pmatrix}
\frac{d_i}{k}\bold{I}_k & B_i^\top \\
B_i^\top & {Q}_i
\end{pmatrix}, \begin{pmatrix}
\bold{I}_k & {P}^\top \\
{P} & {X}
\end{pmatrix} \right \rangle \leq 0 \quad \forall i \in [m] \\
        & \begin{pmatrix}
\bold{I}_k & {P}^\top \\
{P} & {X}
\end{pmatrix} \succeq \bold{0},~\diag(X) = P\bold{1}_k,~X \in \{0,1\}^{n \times n},~P \in \{0,1\}^{n \times k}, 
    \end{aligned} 
\end{align}
which is equivalent to~\eqref{QMP_2}, as shown below.
\begin{theorem} \label{THM:QMP_2_Equivalence} 
    \eqref{ISDP_ML2} is equivalent to \eqref{QMP_2}. 
\end{theorem}
\begin{proof}Let $P$ be feasible for~\eqref{QMP_2} and define $Y \in \{0,1\}^{(n+k)\times (n+k)}$ as
$Y = \begin{psmallmatrix}
        \bold{I}_k \\ P
    \end{psmallmatrix}\begin{psmallmatrix}
        \bold{I}_k \\ P
    \end{psmallmatrix}^\top = \begin{psmallmatrix}
        \bold{I}_k & P^\top  \\
        P & X
    \end{psmallmatrix}$, 
where $X := PP^\top$. Clearly, we have $Y \succeq \bold{0}$ and $X_{ii} = \sum_{j = 1}^k P_{ij}^2 = \sum_{j = 1}^k P_{ij} = \bold{e}_i^\top P\bold{1}_k$ for all~$i \in [n]$, showing that $\diag(X) = P\bold{1}_k$. Moreover, we have
\begin{align*}
    \tr(P^\top Q_i P) + 2 \tr(B_i^\top P) + d_i & = \tr(Q_i X) + 2 \tr(B_i^\top P) + d_i = \left \langle \begin{pmatrix}
\frac{d_i}{k}\bold{I}_k & B_i^\top \\
B_i^\top & {Q}_i
\end{pmatrix}, \begin{pmatrix}
\bold{I}_k & {P}^\top \\
{P} & {X}
\end{pmatrix} \right \rangle 
\end{align*}
for all $i \in [m]$ and $i = 0$. Hence, $X$ and $P$ are feasible for~\eqref{ISDP_ML2} and the objective functions coincide. 

Conversely,  let $P \in \{0,1\}^{n \times k}$ and $X \in \{0,1\}^{n \times n}$ be feasible for~\eqref{ISDP_ML2}. Following the proof of Theorem~\ref{Thm:binPSDrank_lower}, it follows that there exist $x_1, \ldots, x_{k'} \in \{0,1\}^n$ with $k' \geq k$ such that 
$P = [x_1 ~ \dots ~ x_k]$ and~$X = \sum_{j = 1}^{k'}x_jx_j^\top$. 
Since $\diag(X) = P\bold{1}_k$, it follows that for all $i \in [n]$ we have $X_{ii} = \bold{e}_i^\top P \bold{1}_k$ implying that $\sum_{j = 1}^{k'}(x_j)_i^2 = \sum_{j = 1}^k (x_j)_i$.
Since $(x_j)_i \in \{0,1\}$, the equality above only holds if $(x_j)_i = 0$ for all~$j = k+1, \ldots, k'$. As this is true for all $i \in [n]$, we have $x_j = \bold{0}_n$ for all $j = k+1, \ldots, k'$, implying that~${X = \sum_{j=1}^k x_jx_j^\top = PP^\top}$. We can now follow the derivation of the first part of the proof in the converse order to conclude that~$P$ is feasible for~\eqref{QMP_2}. 
\end{proof}
Typical problems that fit in the framework of~\eqref{QMP_2} and~\eqref{ISDP_ML2} are quadratic matrix programs over the packing or partition matrices that require constraints for specific classes, see e.g., Example~\ref{Ex:QMKP}. Another important feature of~\eqref{ISDP_ML2} is that it is possible to impose { a condition on
the rank of $X$.
\Cref{Cor:ExactRank}
implies that if we add the constraint $P^\top \bold{1}_n \geq \bold{1}_k$ to~\eqref{ISDP_ML2}, the resulting matrix $X$ has rank exactly $k$. This makes this formulation suitable for quadratic classification problems that require an exact number of classes, e.g., the (capacitated) max-$k$-cut problem~\cite{GaurEtAl}.

\begin{example}[The quadratic multiple knapsack problem] \label{Ex:QMKP}
Let a set of $n$ items be given, each with a weight~$w_i \in \mathbb{R}_+$ and a profit $p_i \in \mathbb{R}_+$, $i \in [n]$. We are also given a set of $k$ knapsacks, each with a capacity~$c_j \in \mathbb{R}_+$,~$j \in [k]$. Finally, let $R = (r_{i\ell})$ denote a revenue matrix, where $r_{i\ell}$ denotes the revenue of including items $i$ and $\ell$ in the same knapsack. The quadratic multiple knapsack problem (QMKP) aims at allocating each {\color{black}item} to at most one knapsack such that we maximize the total profit of the included items and their interaction revenues, see~\cite{HileyJulstrom}.

Let $P \in \{0,1\}^{n \times k}$ be a packing matrix where $P_{ij} = 1$ if and only if item $i$ is allocated to knapsack~$j$. The capacity constraint can be modeled as~$P^\top w \leq c$, where $w \in \mathbb{R}_+^n$ and $c \in \mathbb{R}^k_+$ denote the vector of item weights and knapsack capacities, respectively. The total profit can be computed as~$\langle R, PP^\top \rangle + p^\top P \bold{1}_k$, where $p \in \mathbb{R}^n$ denotes the vector of item profits. 
It follows from Theorem~\ref{THM:QMP_2_Equivalence} that we can model the QMKP as the following {\color{black}binary} SDP: 
\begin{align}  \label{BSDP_QMKP}
    \begin{aligned}
       {\color{black} \max} \quad & \left \langle \begin{pmatrix}
\bold{0} & \frac{1}{2}\bold{1}_k p^\top \\
\frac{1}{2} p \bold{1}_k^\top & R
\end{pmatrix}, \begin{pmatrix}
\bold{I}_k & {P}^\top \\
{P} & {X}
\end{pmatrix} \right \rangle \\
\text{s.t.}\quad & P^\top w \leq c,~\diag(X) = P\bold{1}_k \\
        & \begin{pmatrix}
\bold{I}_k & {P}^\top \\
{P} & {X}
\end{pmatrix} \succeq \bold{0},~X \in \{0,1\}^{n \times n},~P \in \{0,1\}^{n \times k}.  
    \end{aligned} 
\end{align}

\vspace{-0.8cm}

\hfill \qed
\end{example}

\section{Problem-specific formulations} \label{Sect:ProblemSpecific}
In this section we consider MISDP formulations of problems that do not belong to the binary quadratic problems or for which the reformulation technique differs from the ones in~Section~\ref{Sect:BQPformulations}.

\subsection{The QAP as a MISDP} \label{Sect:QAP}

We present a MISDP formulation of the quadratic assignment problem (QAP) that is derived by a matrix-lifting approach. 
To the best of our knowledge, our QAP formulation provides the most compact convex mixed-integer formulation of the problem in the literature. 
The formulation is motivated by the matrix-lifting SDP relaxations of the QAP derived in~\cite{DingWolkowicz}.

The quadratic assignment problem is an optimization problem of the  following  form:
\begin{align}\label{QAP-integer}
\min_{X \in \Pi_n} \tr (AXBX^\top) + \tr( CX^\top),
\end{align}
where $A,B \in {\mathcal S}^n$, $C\in \mathbb{R}^{n \times n}$
and $\Pi_n$ is the set of $n\times n$ permutation matrices.  
The QAP is among the most difficult $\mathcal{NP}$-hard combinatorial optimization problems to solve in practice. The QAP {\color{black}was} introduced in~1957 by Koopmans
and Beckmann~\cite{KoopmansBeckmann} as a model for location problems. 
Nowadays, the QAP is known as a generic model for various (real-life) problems.

By exploiting properties of the Kronecker product and \Cref{Thm:binPSDrank}, one can lift the QAP into the space of $(n^2 +1) \times (n^2 +1)$ $\{0,1\}$--matrix variables and obtain a BSDP formulation of the QAP, {see Section~\ref{section:vector lifting}}. 
Since this vector-lifting approach results in a problem formulation with a large matrix variable, we consider here a matrix-lifting approach for the QAP. Ding and Wolkowicz~\cite{DingWolkowicz} introduce several matrix-lifting SDP relaxations of the QAP with matrix variables of order $3n$. 
By imposing integrality on the matrix variable~$X$ in one of these SDP relaxations, i.e., the relaxation $\mathit{MSDR}_0$ in~\cite{DingWolkowicz}, we obtain the following MISDP:
\begin{align}  \label{QAP-BSDP}
\begin{aligned}
\min \quad&  \langle A, Y \rangle  + \langle C, X \rangle  \\
\textrm{s.t.} \quad & \begin{pmatrix}
\bold{I}_n & X^\top & R^\top \\
X & \bold{I}_n & Y \\
R & Y & Z
\end{pmatrix}\succeq \mathbf{0},~R = XB \\
& X\in \Pi_n, ~R\in  \mathbb{R}^{n \times n}, ~Y,Z\in {\mathcal S}^n. \end{aligned}
\end{align}
Note that if $B$ is an integer matrix, then $R$ is also an integer matrix. However, we do not have to impose integrality on $R$ explicitly.

The Schur complement lemma implies that the linear matrix inequality in~\eqref{QAP-BSDP} is equivalent to
\begin{align}\label{QAPlmi} 
\begin{pmatrix}
  \bold{I}_n & Y \\
  Y & Z
\end{pmatrix} - \begin{pmatrix}
XX^\top & XR^\top \\
RX^\top & RR^\top 
\end{pmatrix} \succeq \mathbf{0}.
\end{align} 
Now, we are ready to prove the following result.
\begin{proposition} \label{Prop:QAP_TSP}
The MISDP~\eqref{QAP-BSDP} is equivalent to~\eqref{QAP-integer}. 
\end{proposition}
\begin{proof}
Let $(X,Y,Z, R)$ be feasible for \eqref{QAP-BSDP}. 
Then $XX^\top= \bold{I}_n$ and
$\begin{psmallmatrix}
\bold{I}_n- XX^\top & Y -XR^\top \\
Y - RX^\top & Z - RR^\top 
\end{psmallmatrix} \succeq \mathbf{0}$
imply that~${Y =XR^\top}$. Thus, $Y=XB^\top X^\top = XBX^\top$, meaning that the two objectives coincide.

Conversely, let $X$ be feasible for \eqref{QAP-integer}. Define
$R:=XB$, $Y:=XR^\top$ and $Z:=RR^\top.$
It trivially follows that the constraints in \eqref{QAP-BSDP} are satisfied and that the two objective functions coincide.
\end{proof}
Many combinatorial optimization problems can be formulated as the QAP,  see e.g.,~\cite{Burkard}. We provide an example below.
\begin{example}[The traveling salesman problem]\label{example:tsp}
Given is a complete undirected graph~${K_n=(V,E)}$ with~${n:=|V|}$ vertices and a matrix $D=(d_{ij}) \in \mathcal{S}^n$, where~$d_{ij}$ is the cost of edge~$\{ i,j \}\in E$. 
The goal of the traveling salesman problem (TSP) is to find a Hamiltonian cycle of minimum cost in $K_n$. 

Let $B$ be the adjacency matrix of the  {\color{black}tour} on $n$ vertices, i.e., $B$ is a symmetric Toeplitz matrix whose first row is $[0~1~\bold{0}_{n-3}^\top~1]$. It is well-known, {\color{black} see e.g.,~\cite{deKlerkPasechnikSotirov}}, that~\eqref{QAP-integer} with this matrix~$B$ and~$A = D$  is a formulation of  the TSP. 
Therefore, a MISDP formulation of the TSP is the optimization problem~\eqref{QAP-BSDP} where the objective is replaced by $\frac{1}{2}\langle  D, Y \rangle$.

Another MISDP formulation of the TSP is given in~\Cref{Subsec:Association}, see also~\eqref{TSPAlgConnect}. The latter formulation is, to the best of our knowledge, the most compact formulation of the TSP. \hfill \qed
\end{example}

\subsection{MISDP formulations of the graph partition problem}
\label{section:GPP}
We present here various MISDP formulations of the graph partition problem (GPP). Several of the here derived formulations  cannot be obtained by using results from~\Cref{section:vector lifting} and \Cref{section:matrix lifting}.  

The GPP is the problem of partitioning the vertex
set of a graph into a fixed number of sets, say~$k$, of given sizes such that the sum of weights of edges joining different sets is optimized.
If all sets are of equal size, then the corresponding problem is known as the~$k$-equipartition problem ($k$-EP). The case of the GPP with $k = 2$ is known as the graph bisection problem (GBP). 
To formalize, let~$G=(V,E)$ be an undirected graph on $n :=|V|$ vertices and let  $W := (w_{ij}) \in \mathcal{S}^n$ denote a weight matrix with $w_{ij} = 0$ if $\{i,j\} \notin E$.  The graph partition problem aims to partition the vertices of $G$ into $k$  ($2\leq k \leq n-1$) disjoint sets $S_1$, \ldots, $S_k$ of specified sizes $m_1\geq \cdots \geq m_k \geq 1$, $\sum_{j=1}^k m_j=n$ such that the total weight of edges joining different sets $S_j$ is minimized.

For a given partition of $V$ into $k$ subsets,  let $P=(P_{ij}) \in \{0,1\}^{n\times k}$ be the partition matrix, where~${P_{ij}=1}$ if and only if $i\in S_j$ for $i\in[n]$ and $j\in [k]$.
The total weight of the partition equals:
\begin{align}\label{GBP:obj}
\frac{1}{2} \tr \left( W({\mathbf J}_n-PP^\top) \right)  =\frac{1}{2} \tr  (LPP^\top),
\end{align}
where $L:= \textrm{Diag}(W{\mathbf 1}_n) - W$ is the weighted Laplacian matrix of $G$.
The GPP can be formulated as the following quadratic matrix program:
\begin{align}\label{GPP01} &
\begin{aligned}
\min \quad &  \frac{1}{2}  \langle L,PP^\top \rangle \quad \textup{s.t.} \quad  P{\mathbf 1}_k = {\mathbf 1}_n,~ P^\top  {\mathbf 1}_n = {\mathbf m},~ P \in \{0,1 \}^{n\times k},
\end{aligned}
\intertext{where ${\mathbf m}=[m_1~\ldots~m_k]^\top$.
The formulation~\eqref{GPP01} is a special case of the quadratic matrix program~\eqref{QMP_2}. 
Therefore, applying \Cref{THM:QMP_2_Equivalence}, the  {GPP} can be  modeled as follows:}
    & \begin{aligned}
        \min \quad & \frac{1}{2}\left \langle L, X \right \rangle \\
\text{s.t.}\quad  & P \bold{1}_k =  \bold{1}_n,~ P^\top \bold{1}_n = {\mathbf m},~\diag(X) =  \bold{1}_n\\ 
        & \begin{pmatrix}
\bold{I}_k & {P}^\top \\
{P} & {X}
\end{pmatrix} \succeq \bold{0},~X \in \{0,1\}^{n \times n},~P \in \{0,1\}^{n \times k}. 
    \end{aligned} \label{BSDP_GPP} 
\intertext{{\color{black}The doubly nonnegative relaxation of \eqref{BSDP_GPP} is similar to the  relaxation for the $k$-partition problem from \cite{FAIRBROTHER201797}.}
For the $k$-EP and the GBP, we can derive simpler formulations by removing $P$ from the model. \endgraf
In the case of the $k$-EP,  the QMP~\eqref{GPP01} is a special case of~\eqref{QMP_1}, and therefore  $k$-EP can be modeled as follows:}
& \begin{aligned}
\min \quad &  \frac{1}{2} \langle L, X\rangle \\
\textrm{s.t.} \quad & \diag({X}) = {\mathbf 1}_n,~X\bold{1}_n= \frac{n}{k}\bold{1}_n    \\
& k{X} - \bold{J}_n \succeq \mathbf{0},~ X\in {\mathcal S}^n, ~X \in \{ 0, 1 \}^{n\times n}. 
\end{aligned} \label{GEP_SDP}
\end{align} 
This result follows from~\Cref{THM:QMP_1_Equivalence}. An alternative proof is provided below.
\begin{proposition}
Let $\bold{m}=\frac{n}{k}\bold{1}_k$.
Then, the QMP~\eqref{GPP01} for the $k$-EP is equivalent to the BSDP~\eqref{GEP_SDP}.
\end{proposition}
\begin{proof}
Let $P$ be feasible for \eqref{GPP01} where  $\bold{m}=\frac{n}{k}\bold{1}_k$. We define 
$X := PP^\top$. The first and second constraint in~\eqref{GEP_SDP}, as well as $X\in \{ 0, 1 \}^{n\times n}$ follow by direct verification. 
Let ${p_i}$ be the $i$th column of $P$ for  $i\in[k]$, then
$$
kX-\bold{J}_n= k PP^\top - \bold{1}_n\bold{1}_n^\top 
= k \sum_{i=1}^k p_i p_i^\top - \left (  \sum_{i=1}^k p_i \right) \left (  \sum_{i=1}^k p_i \right) ^\top 
=\sum_{i<j} (p_i-p_j)({p_i}- {p_j})^\top \succeq \mathbf{0}.
$$
Conversely, let $X$ be feasible for \eqref{GEP_SDP}. Then, it follows from~\Cref{Thm:binPSD} and \Cref{Prop:binaryProperties}  that  there exist~${x_i \in \{0,1\}^n}$, $i\in [r]$, $k\geq r$ such that  $X= \sum_{i=1}^r x_ix_i^\top$ where $\sum_{i=1}^r x_i=\bold{1}_n$.
Since the constraint~${X\bold{1}_n= \frac{n}{k}\bold{1}_n}$ is invariant under permutation of rows and columns of~$X$, we have that the sum of the elements in each row and column of the block matrix~$\bold{J}_{n_1} \oplus \cdots \oplus \bold{J}_{n_r}$  equals~$n/k$. From this it follows that~$r=k$ and~$\bold{1}_n^\top x_i = n/k$ for $i\in [k]$.
It is easy to verify that~${P:=[x_1 ~\ldots~ x_k ]\in \{0,1\}^{n\times k}}$
is feasible for~\eqref{GPP01}. Since the two objectives coincide, the result follows.
\end{proof}
Next result shows that the MISDP~\eqref{BSDP_GPP} also simplifies for the GBP. It has to be noted, however, that the GBP is not a special case of~\eqref{QMP_1}.
\begin{proposition}
Let  $\bold{m}=[m_1~n-m_1]^\top$, $1\leq m_1 \leq n/2$.
Then, the QMP~\eqref{GPP01} for the GBP is equivalent to the following BSDP:
\begin{align} \label{GBP_SDP}
\begin{aligned}
\min \quad &  
\frac{1}{2} \langle L, X \rangle \\
\textup{s.t.} \quad & \diag(X) = {\mathbf 1}_n,~ \langle \bold{J}_n,  X \rangle = m_1^2 + (n-m_1)^2   \\
& 2X - \bold{J}_n \succeq \mathbf{0}, ~ X\in {\mathcal S}^n, ~  X\in \{ 0, 1 \}^{n\times n}. 
\end{aligned}
\end{align}
\end{proposition}

\begin{proof}
Let ${P}$ be feasible for \eqref{GPP01}. We define 
$X := PP^\top.$ The first and second constraint in \eqref{GBP_SDP} follow by direct verification. 
Let ${p_i}$ be the $i$th column of $P$ for  $i\in[2]$, then
$$
2X-\bold{J}_n=2 PP^\top - \bold{1}_n\bold{1}_n^\top 
= 2 \sum_{i=1}^2 p_i p_i^\top - \left (  \sum_{i=1}^2 p_i \right) \left (  \sum_{i=1}^2 p_i \right) ^\top
= (p_1-p_2)(p_1- p_2)^\top \succeq \mathbf{0}.
$$

Conversely, let $X$ be feasible for \eqref{GBP_SDP}. Then, it follows from~\Cref{Thm:binPSD} and~\Cref{Cor:binPSDrank_upper} that there exist~${x_1, x_2\in \{0,1\}^n}$ such that 
$X= x_1x_1^\top +x_2{x_2}^\top$ where $x_1+x_2={\mathbf 1}_n$. Note that $X$ cannot have rank one or zero for $1\leq m_1< n$.
From $\langle \bold{J}_n,{X} \rangle = m_1^2 + (n-m_1)^2$, it follows that~${\mathbf 1}_n^\top x_1=m_1$ or
 ${\mathbf 1}_n^\top x_1 =n-m_1$.  Without loss of generality, we assume that~${{\mathbf 1}^\top x_1=m_1}$. Clearly,~$P :=[x_1~x_2]$ is feasible for \eqref{GPP01}. Moreover, the two objective functions coincide. 
\end{proof}

In the remainder of this section, we derive yet another alternative MISDP formulation of the GPP, different from~\eqref{BSDP_GPP}. 
For that purpose we notice that the GPP can also be formulated as a QMP of the following form:
\begin{align} \tag{$QMP_3$} \label{QMP3}
    \begin{aligned}
\min  \quad  & \tr(P^\top Q_0 P) + \tr(P C_0 P^\top) + 2\tr(B_0^\top P)+d_0  \nonumber \\
\mbox{s.t.} \quad &  \tr(P^\top Q_i P) + \tr(P C_i P^\top) + 2\tr(B_i^\top P)+d_i \leq 0 \quad \forall i\in[m]\\
& P\in  \mathbb{R}^{n \times k},   \nonumber
\end{aligned}
\end{align} 
where $Q_i \in \mathcal{S}^n$,  $C_i\in \mathcal{S}^k$, $B_i\in \mathbb{R}^{n\times k}$, $d_i \in \mathbb{R}$ for $i=0,1,\ldots, m$. Note that \eqref{QMP_2} is a special case of~\eqref{QMP3}.
Examples of problems that are of this form are quadratic problems with orthogonality constraints, see e.g.,~\cite{AnstreicherWolkowicz}.
The GPP can be formulated as follows, see e.g.,~\cite{10.2307/23011949}:
\begin{align}\label{GPP02}
\begin{aligned}
\min \quad & \frac{1}{2} \langle  L, P P^\top \rangle    \\
\textrm{s.t.} \quad &  P^\top  \bold{1}_n = {\mathbf m},~ P^\top P = \Diag({\mathbf m})  \\
& \diag(P P^\top) = \bold{1}_n,~ P \geq \mathbf{0},  ~P\in  \mathbb{R}^{n \times k}.
\end{aligned}
\end{align}
To reformulate \eqref{GPP02} as a MISDP we introduce matrices $X_1 \in \mathcal{S}^n$ and $X_2\in \mathcal{S}^k$ such that~${X_1=PP^\top}$ and~${X_2= P^\top P}$ and relax these matrix equalities to the linear matrix inequalities (LMIs)~${X_1 - PP^\top \succeq \bold{0}}$ and~${X_2 - P^\top P \succeq \mathbf{0}}$, respectively. These can be rewritten as
$$
\begin{pmatrix}
\mathbf{I}_k & P^\top \\
P & X_1
\end{pmatrix} \succeq \bold{0} \quad \text{and} \quad
\begin{pmatrix}
\mathbf{I}_n & P \\
P^\top & X_2
\end{pmatrix} \succeq \mathbf{0}.
$$
After introducing the constraints  $\diag(X_1)= \bold{1}_n $ and $X_2 = \Diag({\mathbf m})$, we obtain the following MISDP:
\begin{align} \label{GPP03}
\begin{aligned} 
\min \quad&  \frac{1}{2} \langle {L},  X_1 \rangle  \\
\textrm{s.t.} \quad &  P \bold{1}_k = \bold{1}_n,~\diag(X_1)= \bold{1}_n,~
X_2 = \Diag({\mathbf m})  \\
& \begin{pmatrix}
\mathbf{I}_k & P^\top \\
P & X_1
\end{pmatrix} \succeq \mathbf{0},~
 \begin{pmatrix}
\mathbf{I}_n & P \\
P^\top & X_2
\end{pmatrix} \succeq \mathbf{0},~X_1 \in {\mathcal S}^n,~X_2 \in \mathcal{S}^k, ~P\in \{0,1\}^{n\times k}.
\end{aligned}
\end{align}
We prove below that \eqref{GPP03} is an exact formulation of the GPP. 
\begin{proposition}
The MISDP~\eqref{GPP03} is an exact formulation of the GPP.
\end{proposition}
\begin{proof}
We prove the result by showing the equivalence between  \eqref{GPP02}  and \eqref{GPP03}.

Let $P \in \mathbb{R}^{n \times k}$ be feasible for \eqref{GPP02}. 
Then, it follows from $\diag(P P^\top) = \bold{1}_n$ that~$(PP^\top)_{ii}= \sum_{j=1}^k P_{ij}^2=1$ for $i\in [n]$.  From this and $P \geq \mathbf{0}$, we obtain $0\leq P_{ij} \leq 1$ for all~$i\in [n]$, $j\in [k]$.
From~$P^\top  \bold{1}_n = {\mathbf m}$ it follows that $ \sum_{i,j}P_{ij}=n $ and  from $P^\top P = \Diag({\mathbf m})$  that~$\tr(P^\top P)=n$, and thus~$ \sum_{i,j} P^2_{ij}=n$.
Therefore,~$P_{ij}\in \{0,1\}$ for all~$i\in [n]$, $j\in [k]$. The equality $\diag(PP^\top ) = \bold{1}_n$ then implies that $P\bold{1}_k=\bold{1}_n$.
It follows from the discussion prior to the proposition that $X_1:=P P^\top$ and~$X_2 := P^\top P$ are feasible for~\eqref{GPP03}.

Conversely, let $X_1$, $X_2$ and $P$ be feasible for  \eqref{GPP03}.
From  $P\in \{0,1\}^{n\times k}$ and $P\bold{1}_k=\bold{1}_n$ it follows that 
$\diag(PP^\top)= \bold{1}_n$. 
From $X_2 - P^\top P \succeq \mathbf{0} $ and
$\bold{1}_k^\top(X_2 - P^\top P )\bold{1}_k=0$ it follows that~${(X_2 - P^\top P )\bold{1}_k=\mathbf{0}}$ and thus $P^\top \bold{1}_n = {\mathbf m}.$
Moreover, we have $(P^\top P)_{ii}= \sum_{j=1}^n P_{ji}^2=  \sum_{j=1}^n P_{ji} =m_i$ for $i\in [k]$, implying that $\diag(P^\top P) = {\mathbf m}$.
{Finally, we have~$X_2 - P^\top P = \Diag({\mathbf m}) - P^\top P \succeq \mathbf{0}$, where it follows from above that the latter matrix has a diagonal of zeros.} Thus, we must have $P^\top P = \Diag({\mathbf m})$, which concludes the proof.
\end{proof}
The MISDP~\eqref{GPP03} has two LMIs and requires integrality constraints only on a matrix of size $n\times k$, while \eqref{BSDP_GPP} has only one LMI and asks for integrality on matrices of size~$n\times n$ and~$n\times k$.

\subsection{MISDP formulations via association schemes} \label{Subsec:Association}
Association schemes provide a unifying framework 
for the treatment of problems 
in several branches of mathematics, including algebraic graph theory and coding theory.
Delsarte~\cite{Delsarte} was the first to combine association schemes and linear programming for codes.  Schrijver~\cite{Schrijver2005NewCU}  refined the Delsarte bound by using semidefinite programming.
De Klerk et al.~\cite{KelrkOliveiraPasechnik} introduce a framework for deriving SDP relaxations of  optimization problems on graphs by using association schemes. 
 In \Cref{sect:AssocSchemeTSP} (resp.~\Cref{sect:AssocSchemeEP}) 
we derive  MISDP formulations of the TSP (resp.~$k$-EP) by exploiting   association schemes.

An association scheme of rank $r$ is a set $\{A_0,A_1,\ldots, A_r \} \subseteq \mathbb{R}^{n\times n}$ of $\{0,1\}$--matrices that satisfy the following properties:
\begin{enumerate}
\item[$(i)$] $A_0=\mathbf{I}_n$ and $\sum_{i=0}^r A_i=\mathbf{J}_n$;
\item[$(ii)$] $A^{\top}_i\in \{A_0, A_1,\ldots, A_r \} $ for $i \in \{0,1,\ldots,r\}$;
\item[$(iii)$] $A_iA_j=A_jA_i$ for $i,j\in \{0,1,\ldots,r\}$;
\item[$(iv)$] There exist {\color{black} scalars} $p^h_{ij}$ such that $A_iA_j=\sum_{h=0}^r p^h_{ij}A_h$ for $i,j\in\{0,1,\ldots,r\}$. 
\end{enumerate}
The numbers $p^h_{ij}$  are called the intersection numbers of the association scheme.  For more background on association schemes, we refer to~\cite{BrouwerHaemers,Godsil93}.

We restrict here to symmetric association schemes, i.e., we assume that all matrices $A_i$ are symmetric.
Matrices $A_i$ are linearly independent and they generate a commutative $(r+1)$-dimensional algebra of symmetric matrices.  This algebra is called the Bose-Mesner algebra of the association scheme. 
Since the matrices $A_i$ for all ${i=0,1,\ldots, r}$ commute, they can be  simultaneously diagonalized. Moreover, there exists a unique basis for the Bose-Mesner algebra consisting of minimal idempotents~$E_i$,~${i=0,1,\ldots ,r}$. These matrices satisfy $E_i E_j = \delta_{ij}E_i$ for all $i,j\in \{0,1,\ldots ,r\}$ and~${\sum_{i=0}^r E_i = {\bold I}_n}$,
where $\delta_{ij}$ is the Kronecker delta function. We assume without loss of generality~that~$E_0=\frac{1}{n}\bold{J}_n$. Bases~$\{A_0,A_1,\ldots, A_r\}$ and~$\{E_0,E_1,\ldots, E_r\}$  are related by: 
\begin{align}\label{basisAssocSch}
A_j = \sum_{i=0}^r P_{ij}E_i 
\quad \mbox{and} \quad
E_j = \frac{1}{n}\sum_{i=0}^r Q_{ij}A_i \quad j=0,1,\ldots, r.     \end{align}
Parameters~$P_{ij}$ and~$Q_{ij}$ for $i,j=0,1,\ldots,r$ are known as the eigenvalues of the association scheme and~${P=(P_{ij})_{i,j=0}^r}$ and~${Q=(Q_{ij})_{i,j=0}^r}$ are the eigenmatrices of
the association scheme.

Distance-regular graphs are closely related to association schemes.
Classes of distance-regular graphs include complete graphs,  complete bipartite graphs and cycle graphs.
Let us define the {distance-$i$ graph} $\Gamma_i$ as the graph with vertex set
$V$, in which two vertices~$x$ and $y$ are adjacent if and only if the length of the shortest path between $x$ and $y$ in~$G$ is~$i$.
The adjacency matrix $A_i$ of $\Gamma_i$
is called the {distance-$i$ matrix} of $G$ for $i = 0, 1, \ldots, d$.
It is known that the set of so generated  matrices~$\{A_0,A_1,\ldots, A_r \}$ form an association scheme.
It follows from the construction that this association scheme forms a (regular) partition of the edge set of a complete graph.

If $\{A_0,A_1,\ldots, A_r \}$ is an association scheme and  $P \in \Pi_n$, then  $\{P^{\top}A_0P, P^{\top}A_1 P,\ldots, P^{\top} A_r P \}$  is also an association scheme. Let us define the following polytope:
\begin{align}\label{Assoc:polytope}
{\mathcal H}(A_0,A_1,\ldots, A_r)  := \mbox{conv} \left \{ (P^{\top}A_0P, P^{\top}A_1 P,\ldots, P^{\top} A_r P)  \, :\, P\in \Pi_n\right \}.    
\end{align}
The vertices of ${\mathcal H}(A_0,A_1,\ldots, A_r)$ are of the form $\left \{ P^{\top}A_0P, P^{\top}A_1 P,\ldots, P^{\top} A_r P \right \}$ with~${P\in \Pi_n}$.
De Klerk et al.~\cite{KelrkOliveiraPasechnik}  show that many combinatorial optimization problems may be formulated as optimization problems over the polytope~\eqref{Assoc:polytope}. The key is to formulate a combinatorial optimization problem as the problem of finding an {\color{black}optimal} weight distance-regular subgraph of a weighted complete graph.
Since the distance matrices of a distance-regular graph form an association scheme, the  problem can then be reformulated as an optimization problem over the polytope~\eqref{Assoc:polytope} for an appropriate association scheme. 

For a given association scheme $\left \{A_0,A_1,\ldots, A_r \right \}$, De Klerk et al.~\cite{KelrkOliveiraPasechnik}  approximate~${\mathcal H}(A_0,A_1,\ldots, A_r)$ 
by the following larger set:
\begin{align}\label{AssocSch:Relax}
{\mathcal A} := 
\bigg \{
(X_0,\ldots,X_r)   : & \,  
X_0= \mathbf{I}_n, \sum\limits_{i=0}^r X_i, =\mathbf{J}_n,  \,
\sum\limits_{i=0}^r Q_{ij} X_i \succeq \mathbf{0}, j \in [r],   X_i \geq \mathbf{0}, X_i=X_i^{\top},  i\in [r]
 \bigg \},
\end{align}
where $Q_{ij}$ are the dual eigenvalues of the association scheme $\{A_0,A_1,\ldots, A_r \} $. They exploit the set~${\mathcal A}$ to derive SDP relaxations of various combinatorial optimization problems. 
The LMIs in ${\mathcal A}$ follow from the properties of the  basis of minimal idempotents, see~\eqref{basisAssocSch} and \cite{KelrkOliveiraPasechnik} for details.
In the sequel,  we show that by imposing integrality to {\color{black}only one} of the matrices in~\eqref{AssocSch:Relax}, one obtains an exact problem formulation.

\subsubsection{The traveling salesman problem}
\label{sect:AssocSchemeTSP} 
Let us reconsider the traveling salesman problem (TSP), see~Example~\ref{example:tsp}. In this section we assume that $n$  is {odd}. The results for $n$ even can be derived similarly. 

The adjacency matrix of a Hamiltonian cycle is  a symmetric circulant matrix that belongs to the association scheme of symmetric circulant matrices, known as the Lee scheme~\cite{LeeScheme}.
The dual eigenvalues of the Lee scheme are
$Q_{0j}=2$ for $j\in [r]$ and $Q_{i0}=1$ for $i=0,1,\ldots,r$,
where $r = \lfloor {n}/{2} \rfloor$.
After incorporating those dual eigenvalues in~\eqref{AssocSch:Relax}, the authors of~\cite{KelrkOliveiraPasechnik,deKlerkPasechnikSotirov} derive an SDP relaxation for the TSP. 
By imposing integrality on the variable $X_1$ in the SDP relaxation from~\cite{KelrkOliveiraPasechnik,deKlerkPasechnikSotirov} we obtain the following MISDP:
\begin{align}\label{TSPAssocMIP}
\begin{aligned}
\min &   \quad \frac{1}{2} \langle D, X_1 \rangle   \\
{\rm s.t.} & \quad \mathbf{I}_n  + \sum_{i=1}^r X_i = \mathbf{J}_n,~ \mathbf{I}_n  + \sum_{i=1}^r \cos \left (\frac{2ij\pi}{n} \right ) X_i  \succeq \mathbf{0} ~~ \forall j \in [r]  \\
& \quad X_1 = X_1^{\top} \in \{0,1\}^{n\times n},~ X_i\geq \mathbf{0}, ~ X_i \in {\mathcal S}^n ~~\forall i= 2,\ldots, r,
\end{aligned}
\end{align}
where $r=  \lfloor {n}/{2} \rfloor$.
Next, we show that the MISDP~\eqref{TSPAssocMIP} is an exact model of the TSP.
\begin{theorem}
The MISDP~\eqref{TSPAssocMIP} is an exact formulation of the TSP.
\end{theorem}

\begin{proof}
To prove the result we use the following  ISDP formulation of the TSP by Cvetkovi\'c~et~al.~\cite{Cvetkovic}:
\begin{align}\label{TSPAlgConnect}
\begin{aligned}
\min \quad &   \quad \frac{1}{2} \langle D, X \rangle   \\
{\rm s.t.} & \quad 2 \mathbf{I}_n - X + 2 \left ( 1- \cos \left (\frac{2\pi}{n} \right ) \right ) ( \mathbf{J}_n - \mathbf{I}_n )  \succeq \mathbf{0}  &  \\
& \quad X \mathbf{1}_n = 2 \mathbf{1}_n,~\diag(X)=\mathbf{0},~ X \in {\mathcal S}^n, \,\,  X\in \{0,1\}^{n\times n}.
\end{aligned}
\end{align}
The ISDP~\eqref{TSPAlgConnect} is derived by exploiting the algebraic connectivity of a cycle on $n$ vertices.

Let $X$ be feasible for~\eqref{TSPAlgConnect}. Then, $X$ is the  adjacency matrix of a Hamiltonian cycle, hence $X$ is a symmetric circulant matrix.
Define $X_0:= \mathbf{I}_n$, $X_1:=X$ and $X_2 := X_{1}^2-2 \mathbf{I}_n$, and the recurrence relation~$X_1 X_i=X_{i+1} +X_{i-1}$  for $i=2,\ldots, r-1.$ Then, the set $ \{ X_0, X_1,\ldots, X_r \}$ where $r=\lfloor n/2\rfloor$ forms an association scheme of (permuted) symmetric circulant matrices. {\color{black}Note that  $X_i$ $i=0,1,\ldots,r$ are the distance matrices of the Hamiltonian cycle whose adjacency matrix is $X_1$.} The LMIs in~\eqref{TSPAssocMIP}   follow from the fact that the corresponding basis of minimal idempotents $\{E_0,\ldots, E_r\}$ satisfies
$E_j = \frac{1}{n}\sum_{i=0}^r Q_{ij}X_i \succeq \mathbf{0}$,
for~$j=0,1,\ldots, r$. See also~\cite{deKlerkPasechnikSotirov} for a direct proof of the validity of these LMIs that exploits the fact that the  $X_i$'s may be simultaneously diagonalized.

Conversely, let $X_i$ for $i\in [r]$ be feasible for~\eqref{TSPAssocMIP}. Define  $X:=X_1$. Now, similarly as in~\cite[Theorem 4.1]{deKlerkPasechnikSotirov} 
one can show that  $2 \mathbf{I}_n +  \left ( 1- 2\cos \left (2\pi/{n} \right ) \right )X_1 +  \left ( 2- 2\cos \left (2\pi/{n} \right ) \right ) \sum_{i=2}^r  X_i  \succeq \mathbf{0}$ may be obtained 
as a nonnegative linear combination of $r$ LMIs from~\eqref{TSPAssocMIP} and $\mathbf{I}_n +  \sum_{i=1}^r X_i  \succeq \mathbf{0}$. 
From~\cite[Theorem 7.3]{KelrkOliveiraPasechnik}, it follows that $X \mathbf{1}_n = 2 \mathbf{1}_n.$ Since the two objectives clearly coincide, the result follows.
\end{proof}

\subsubsection{The \boldmath $k$-equipartition problem}
\label{sect:AssocSchemeEP}
In \Cref{section:GPP} we derived various MISDP formulations of the GPP, including the special case of the $k$-EP.
We now derive an alternative formulation of the $k$-EP by exploiting an appropriate association scheme.

Let $n,k,m\in {\mathbb Z}_+$, and 
$D$ be a nonnegative symmetric matrix of order $n$, where $n=mk$. The $k$-EP can be equivalently formulated as finding a complete regular $k$-partite subgraph on $n$ vertices in $K_{mk}$ of minimum weight. The complete regular $k$-partite graph is strongly regular. The dual eigenvalues of the corresponding association  scheme are given in the following matrix:
\begin{align}\label{kEPdualeig}
Q =  \begin{pmatrix}
        1 & (m-1)k & k-1 \\
        1 & 0 &  -1\\
        1 & -k & k-1
    \end{pmatrix}.
\end{align}
De Klerk et al.~\cite{KelrkOliveiraPasechnik} combine \eqref{AssocSch:Relax} and
\eqref{kEPdualeig} to derive an association scheme-based SDP relaxation of the~$k$-EP.
By imposing integrality  on the matrix variable $X_2$ in the SDP relaxation from~\cite{KelrkOliveiraPasechnik} we obtain:
\begin{align}\label{EPAssoc}
\begin{aligned}
\min & \quad \langle D, X_1\rangle \\
{\rm s.t.}  & \quad (m-1)\mathbf{I}_n - X_2 \succeq \mathbf{0},~ (k-1)\mathbf{I}_n -X_1 + (k-1)X_2 \succeq \mathbf{0}\\
& \quad \mathbf{I}_n + X_1+X_2 = \mathbf{J}_n,~  X_1, X_2 \in {\mathcal S}^n, \,\, X_1 \geq \mathbf{0}, \,\, X_2 \in \{0,1\}^{n \times n}. 
\end{aligned}
\end{align}
In the sequel we show that the MISDP~\eqref{EPAssoc} is an exact formulation of the $k$-EP. Note that in~\eqref{EPAssoc} it is sufficient to require that 
$X_2 \in \{0,1\}^{n \times n}$, since the integrality of $X_1$ follows from ${\mathbf{I}_n + X_1+X_2 = \mathbf{J}_n}$.
\begin{proposition}\label{assoc:kEP}
The BSDP~\eqref{GEP_SDP} is equivalent to the MISDP~\eqref{EPAssoc}.
\end{proposition}
\begin{proof}
Let $X_1$, $X_2$  be feasible for \eqref{EPAssoc}. 
We define matrix $X:=\mathbf{I}_n +X_2$, and show that it is feasible for~\eqref{GEP_SDP}. It follows trivially that  $\diag(X)=\mathbf{1}_n$. Further,   
$X\mathbf{1}_n = (\mathbf{I}_n+X_2)\mathbf{1}_n =(\mathbf{I}_n +A_2)\mathbf{1}_n = m\mathbf{1}_n $ follows from the fact that $X_2 \mathbf{1}_n = A_2 \mathbf{1}_n$, see~\cite[Theorem 7.3]{KelrkOliveiraPasechnik}, where $A_2 := \Diag(\mathbf{1}_k) \otimes (\mathbf{J}_m - \mathbf{I}_m)$. 
To verify the PSD constraint we proceed as follows:
$$  (k-1)\mathbf{I}_n - X_1 + (k-1)X_2 = k(\mathbf{I}_n + X_2) - (\mathbf{I}_n +X_1+X_2)=  kX - \mathbf{J}_n     \succeq \mathbf{0}. $$

Conversely, let $X$ be feasible for  \eqref{GEP_SDP}.
Define $X_1:= \mathbf{J}_n-X$ and $X_2 :=X-\mathbf{I}_n$, from where it follows that $X_1, X_2 \in \{0,1\}^{n\times n}$. 
Then, the first constraint in \eqref{EPAssoc} is trivially satisfied and
$$
  (k-1)\mathbf{I}_n -X_1 + (k-1)X_2 =
  (k-1)\mathbf{I}_n -\mathbf{J}_n + X + (k-1)(X-\mathbf{I}_n) = kX- \mathbf{J}_n \succeq \mathbf{0}.
$$
For the second LMI, we have
$(m-1)\mathbf{I}_n - X_2 = m\mathbf{I}_n -X  \succeq \mathbf{0}, $
where we use the fact that the spectral radius $\rho(X) \leq \| X\|_{\infty} = m$.
The two objective functions clearly coincide.
\end{proof}
{\color{black}Note that the MISDP~\eqref{EPAssoc} has two LMIs and the BSDP~\eqref{GEP_SDP} only one.}
Moreover, observe that
one may replace in~\eqref{EPAssoc} the constraint $ (m-1)\mathbf{I}_n - X_2 \succeq \mathbf{0}$  by $X_2\mathbf{1}_n = (m-1)\mathbf{1}_n$ and obtain a MISDP for the $k$-EP with only one PSD constraint. This follows directly from the proof of Proposition~\ref{assoc:kEP}.

\subsection{MISDP formulations beyond binarity}
Almost all problem formulations that have been discussed before involve matrix variables with entries in~$\{0, 1\}$. In this section we consider several problems that allow for semidefinite formulations where (some of) the variables are integers, but not necessarily restricted to~$\{0,1\}$. 

\begin{example}[The integer matrix completion problem] \label{Ex:IMCP}
A well-known problem in data analysis is the problem of low-rank matrix completion. Suppose a partially observed data matrix is given, i.e., let~$\Omega \subseteq [n] \times [m]$ denote the set of observed entries and let $D \in \mathbb{R}^{n \times m}$ denote a given data matrix that has its support on~$\Omega$. The goal of the low-rank matrix completion problem is to find a minimum rank matrix $X \in \mathbb{R}^{n \times m}$ such that $X$ coincides with $D$ on the set $\Omega$, e.g., see~\cite{Netflix}. 

Since minimizing $\rank(X)$ leads to a nonconvex and therefore hard problem, a related but tractable alternative is given by minimizing the nuclear norm of $X$, i.e.,~${||X||_* := \sum_{i = 1}^n  \sigma_i(X)}$, where $\sigma_i$ denotes the $i$th  singular value of $X$. Hence, we consider the following program:
\begin{align*}
    \min_{X \in \mathbb{R}^{n \times m}}  & \quad||X||_* \quad \text{s.t.} \quad  X_{ij} = D_{ij} \quad \text{for all } (i,j) \in \Omega. 
\intertext{As shown by Recht et al.~\cite{RechtEtAl},  the optimization problem above is equivalent to the following SDP:} 
    \min \quad & \quad \langle \bold{I}_n, Z_1 \rangle + \langle \bold{I}_m, Z_2 \rangle  \\
    \text{s.t.} \quad & \quad   \begin{pmatrix}
        Z_1 & X \\
        X^\top & Z_2
    \end{pmatrix} \succeq \bold{0}, ~X_{ij} = D_{ij} \quad \text{for all } (i,j) \in \Omega.
\end{align*} 
\textcolor{black}{Since $X$ can be thought of as a multiplicative model underlying the data observed in $D$, see~\cite{Netflix}, a possible generalization would be to require the entries in $X$ to be integer.}
This yields the following integer matrix completion problem: 
\begin{align*} 
    \min \quad & \langle \bold{I}_n, Z_1 \rangle + \langle \bold{I}_m, Z_2 \rangle  \\
    \text{s.t.} \quad & X_{ij} = D_{ij} ~~ \text{for all } (i,j) \in \Omega,\quad X_{ij} \in S ~~ \text{for all } (i,j) \notin \Omega \\
    & \begin{pmatrix}
        Z_1 & X \\
        X^\top & Z_2
    \end{pmatrix} \succeq \bold{0}. 
\end{align*}
where $S \subseteq \mathbb{Z}$ is a discrete set. \hfill \qed
\end{example}

\begin{example}[The sparse integer least squares problem] \label{Ex:SILS}
In the integer least squares problem we are given a matrix $M \in \mathbb{R}^{n \times k}$ and a column $b \in \mathbb{R}^n$ and we seek the closest point to $b$ in the lattice spanned by the columns of $M$. Del Pia and Zhou~\cite{PiaZhou} consider the related sparse integer least squares (SILS) problem, which can be formulated as
\begin{align}
    & \begin{aligned}
        \min \quad & \frac{1}{n} || Mx - b ||_2^2 \quad \text{s.t.} \quad  x \in \{0,\pm 1\}^k,~||x||_0 \leq K. 
    \end{aligned} \label{Prob:SILS}
\intertext{The SILS problem has applications in, among other, multiuser detection and sensor networks, see~\cite{PiaZhou} and the references therein. Now, consider the following ternary SDP:}
& \begin{aligned}
        \min \quad & \frac{1}{n} \left\langle \begin{pmatrix}
            1 & x^\top \\ x & X
        \end{pmatrix}, \begin{pmatrix}
            b^\top b  & - b^\top M \\ - M^\top b & M^\top M
        \end{pmatrix} \right \rangle \\
        \text{s.t.} \quad & \tr(X) \leq K,~~ \diag(X) = y_1 + y_2,~~x = y_1 - y_2 \\
        & \begin{pmatrix}
            1 & x^\top \\ x & X
        \end{pmatrix} \succeq \bold{0}, ~~ \begin{pmatrix}
            1 & x^\top \\ x & X
        \end{pmatrix} \in \{0, \pm 1\}^{(k+1) \times (k+1)},~~ y_1, y_2 \in \mathbb{R}^n_+.
    \end{aligned} \label{SILS_TSDP}
\end{align}
It is easy to verify that if $x$ is a solution to~\eqref{Prob:SILS}, then $x$, $X = xx^\top$, $y_1 = \max(x,\bold{0})$ and~$y_2 = \max(-x,\bold{0})$ is feasible for~\eqref{SILS_TSDP} with the same objective value. 
Conversely, if $(x,X,y_1,y_2)$ is feasible for~\eqref{SILS_TSDP}, it follows from~Proposition~\ref{prop:pm1b} that $x = y_1 - y_2$ is a solution to~\eqref{Prob:SILS} with the same objective value.
\end{example}

\section{Conclusions} \label{Sect:Conclusion}
In this paper we {\color{black}showed} that the class of mixed-integer semidefinite programs embodies a rich structure, allowing for compact formulations of many well-known discrete optimization problems. 
Due to the recent progress in computational methods for solving MISDPs~\cite{GallyEtAl,GravitySDP, HojnyPfetsch,KobayashiTakano,MatterPfetsch, deMeijerSotirovCG}, these formulations can be exploited to obtain alternative methods for solving the problems to optimality. 

As most problems are naturally encoded using binary or ternary variables, we {\color{black} started} our research with a study on the general theory related to PSD $\{0,1\}$--, $\{\pm 1\}$-- and $\{0, \pm 1\}$--matrices. Section~\ref{Sect:TheoryPSD} provides a comprehensive overview on this matter, including known and new results. In particular, we {\color{black}presented} a combinatorial, polyhedral, set-completely positive and integer hull description of the set of PSD $\{0,1\}$--matrices bounded by a certain rank, see~Section~\ref{Subsect:theory01}. Several of these results are extended to matrices having entries in $\{\pm 1\}$ and $\{0,\pm 1\}$. 

Based on these matrix results, in particular Theorem~\ref{Thm:binPSD}--\ref{Thm:binPSDrank} and Corollary~\ref{Cor:ExactRank}, we {\color{black}derived} a generic approach to model binary quadratic problems as BSDPs. We {\color{black}derived} a BSDP for the class of binary quadratically constrained quadratic programs, see~\eqref{BQP01_bin}, and for two types of binary quadratic matrix programs, see~\eqref{ISDP_ML1} and~\eqref{ISDP_ML2}. These results are widely applicable to a large number of discrete optimization problems, see also the examples in Section~\ref{Sect:BQPformulations}. 

We moreover we {\color{black}considered} problem-specific MISDP formulations that are derived in a different way than through this generic approach. We {\color{black}provided} compact MISDP formulations of the {\color{black}QAP}, see~\eqref{QAP-BSDP}, and various variants of the {\color{black}GPP}, see~\eqref{BSDP_GPP}, \eqref{GEP_SDP} and~\eqref{GBP_SDP}. We {\color{black}derived} several MISDP formulations of discrete optimization problems that can be modeled using association schemes, see~Section~\ref{Subsec:Association}. We also {\color{black}considered} problems that have discrete but non-binary variables, e.g., the integer matrix completion problem and the sparse integer least squares problem, see Example~\ref{Ex:IMCP} and~\ref{Ex:SILS}, respectively.

Given the wide range of discrete optimization problems for which we derived new formulations based on mixed-integer semidefinite programming, we expect more problems to allow for such representations. It is also interesting to study the behaviour of MISDP solvers on the presented formulations to see whether this leads to competitive solution approaches for the considered problems. 

\vspace{-0.3cm}

{ \normalsize

\bibliography{QCCP_QTSP} 

\providecommand{\noopsort}[1]{}
\begin{thebibliography}{10}

\bibitem{Netflix}
{acm sigkdd} and netflix.
\newblock Proceedings of {KDD} {C}up and {W}orkshop.
\newblock Available at
  \url{http://www.cs.uic.edu/~liub/KDD-cup-2007/proceedings.html}, 2007.

\bibitem{AnjosHandbook}
M.F. Anjos and J.B. Lasserre.
\newblock {\em Handbook on semidefinite, conic and polynomial optimization},
  volume 166.
\newblock Springer Science \& Business Media, 2011.

\bibitem{AnjosWolkowicz}
M.F. Anjos and H.~Wolkowicz.
\newblock Strengthened semidefinite relaxations via a second lifting for the
  max-cut problem.
\newblock {\em Discrete Appl. Math.}, 119(1--2):79--106, 2002.

\bibitem{AnstreicherWolkowicz}
K.~Anstreicher and H.~Wolkowicz.
\newblock On {L}agrangian relaxation of quadratic matrix constraints.
\newblock {\em SIAM J. Matrix Anal. Appl.}, 22(1):41--55, 2000.

\bibitem{BandeltEtAl}
H.J. Bandelt, M.~Oosten, J.H.G.C. Rutten, and F.C.R. Spieksma.
\newblock Lifting theorems and facet characterization for a class of clique
  partitioning inequalities.
\newblock {\em Oper. Res. Lett.}, 24:235--243, 1999.

\bibitem{Beck}
A.~Beck.
\newblock Quadratic matrix programming.
\newblock {\em SIAM J. Optim.}, 17:1224--1238, 2007.

\bibitem{Bell}
E.T. Bell.
\newblock Exponential numbers.
\newblock {\em Amer. Math. Monthly}, 41:411--419, 1934.

\bibitem{BensonSaglam}
H.Y. Benson and {\"U}.~Sa{\u{g}}lam.
\newblock Mixed-integer second-order cone programming: A survey.
\newblock In {\em Theory Driven by Influential Applications}, pages 13--36.
  INFORMS, 2013.

\bibitem{BermanXu}
A.~Berman and C.~Xu.
\newblock \{0,1\} completely positive matrices.
\newblock {\em Linear Algebra Appl.}, 399:35--51, 2005.

\bibitem{BomzeGable}
I.M. Bomze and M.~Gable.
\newblock Optimization under uncertainty and risk: Quadratic and copositive
  approaches.
\newblock {\em Eur. J. Oper. Res.}, 310:449--476, 2023.

\bibitem{BrouwerHaemers}
A.E. Brouwer and W.H. Haemers.
\newblock {\em Handbook of Combinatorics}, chapter Association schemes, pages
  747--771.
\newblock Elsevier Science, Amsterdam, 1995.

\bibitem{BulhoesEtAl}
T.~Bulh{\~{o}}es, A.~Pessoa, F.~Protti, and E.~Uchoa.
\newblock On the complete set packing and set partitioning polytopes:
  Properties and rank 1 facet.
\newblock {\em Oper. Res. Lett.}, 46:389--392, 2018.

\bibitem{Burer2009}
S.~Burer.
\newblock On the copositive representation of binary and continuous nonconvex
  quadratic programs.
\newblock {\em Math. Program.}, 120:479--495, 2009.

\bibitem{Burkard}
R.~Burkard, M.~Dell'Amico, and S.~Martello.
\newblock {\em Assignment Problems}.
\newblock Society for Industrial and Applied Mathematics, Philadelphia, PA,
  USA, 2009.

\bibitem{CerveiraEtAl}
A.~Cerveira, A.~Agra, F.~Bastos, and J.~Gromicho.
\newblock A new branch and bound method for a discrete truss topology design
  problem.
\newblock {\em Comput. Optim. Appl.}, 54(1):163--187, 2013.

\bibitem{Chagas}
V.G. Chagas.
\newblock {\em Exact Algorithms for the Quadratic Bin Packing Problem}.
\newblock PhD thesis, Universidade Estadual de Campinas, 2021.

\bibitem{ChopraRao}
S.~Chopra and M.R. Rao.
\newblock The partition problem.
\newblock {\em Math. Program.}, 59:87--115, 1993.

\bibitem{Cvetkovic}
D.~Cvetkovi\'c, M.~\v{C}angalovi\'c, and V.~Kova\v{c}evi\'c-Vuj\v{c}i\'c.
\newblock Semidefinite programming methods for the symmetric traveling salesman
  problem.
\newblock In G.~Cornu\'jols, R.E. Burkard, and G.J. Woeginger, editors, {\em
  Integer programming and Combinatorial Optimization (IPCO 1999)}, volume 1610
  of {\em Lecture Notes in Computer Science}. Springer, Berlin, Heidelberg,
  1999.

\bibitem{Delsarte}
P.~Delsarte.
\newblock An algebraic approach to the association schemes of coding theory.
\newblock {\em Philips Research Reports Suppl}, 10, 1973.

\bibitem{10.2307/23011949}
Y.~Ding, D.~Ge, and H.~Wolkowicz.
\newblock On equivalence of semidefinite relaxations for quadratic matrix
  programming.
\newblock {\em Math. Oper. Res.}, 36(1):88--104, 2011.

\bibitem{DingWolkowicz}
Y.~Ding and H.~Wolkowicz.
\newblock A low-dimensional semidefinite relaxation for the quadratic
  assignment problem.
\newblock {\em Math. Oper. Res.}, 34(4):1008--1022, 2009.

\bibitem{Duarte}
B.P.M. Duarte.
\newblock Exact optimal designs of experiments for factorial models via
  mixed-integer semidefinite programming.
\newblock {\em Math.}, 11(4):854, 2023.

\bibitem{DukanovicRendl}
I.~Dukanovic and F.~Rendl.
\newblock Semidefinite programming relaxations for graph coloring and maximal
  clique problems.
\newblock {\em Math. Program.}, 109:345--365, 2007.

\bibitem{Eissenblat}
A.~Eisenbl\"atter.
\newblock {\em Frequency assignment in GSM networks: Models, heuristics, and
  lower bounds}.
\newblock PhD thesis, 2001.

\bibitem{FAIRBROTHER201797}
Jamie Fairbrother and Adam~N. Letchford.
\newblock Projection results for the k-partition problem.
\newblock {\em Discrete Optim.}, 26:97--111, 2017.

\bibitem{GallyPfetsch}
T.~Gally and M.E. Pfetsch.
\newblock Computing restricted isometry constants via mixed-integer
  semidefinite programming.
\newblock Optimization Online:
  \url{http://www.optimization-online.org/DB_FILE/2016/04/5395.pdf}, 2016.

\bibitem{GallyEtAl}
T.~Gally, M.E. Pfetsch, and S.~Ulbrich.
\newblock A framework for solving mixed-integer semidefinite programs.
\newblock {\em Optim. Methods Softw.}, 33(3):594--632, 2018.

\bibitem{GaurEtAl}
D.R. Gaur, R.~Krishnamurti, and R.~Kohli.
\newblock The capacitated max $k$-cut problem.
\newblock {\em Math. Program.}, 115:65--72, 2008.

\bibitem{GilGonzalezEtAl}
W.~Gil-Gonz{\'a}lez, A.~Molina-Cabrera, O.D. Montoya, and L.F.
  Grisales-Nore{\~n}a.
\newblock An {MI}-{SDP} model for optimal location and sizing of distributed
  generators in dc grids that guarantees the global optimum.
\newblock {\em Appl. Sci.}, 10(21):7681, 2020.

\bibitem{Godsil93}
C.D. Godsil.
\newblock {\em Algebraic Combinatorics}.
\newblock Chapman \& Hall, New York, 1993.

\bibitem{GoemansWilliamson}
M.~X. Goemans and D.~P. Williamson.
\newblock Improved approximation algorithms for maximum cut and satisfiability
  problems using semidefinite programming.
\newblock {\em J. ACM}, 42(6):1115--1145, 1995.

\bibitem{GrotschelEtAl}
M.~Gr\"otschel, L.~Lov\'asz, and A.~Schrijver.
\newblock {\em Geometric algorithms and combinatorial optimization}, volume~2.
\newblock Springer, Berlin, 1988.

\bibitem{GrotschelWakabayashi}
M.~Gr\"otschel and Y.~Wakabayashi.
\newblock Facets of the clique partitioning polytope.
\newblock {\em Math. Program.}, 47:367--387, 1990.

\bibitem{Helmberg1997FixingVI}
C.~Helmberg.
\newblock Fixing variables in semidefinite relaxations.
\newblock {\em SIAM J. Matrix Anal. Appl.}, 21(2):952--–969, 2000.

\bibitem{HelmbergThesis}
C.~Helmberg.
\newblock {\em Semidefinite programming for combinatorial optimization}.
\newblock habilitation, Germany, 2000.

\bibitem{GravitySDP}
H.~Hijazi, G.~Wang, and C.~Coffrin.
\newblock Gravity{SDP}: A solver for sparse mixed-integer semidefinite
  programming, 2018.
\newblock Available at
  \url{https://github.com/coin-or/Gravity/tree/GravitySDP}.

\bibitem{HileyJulstrom}
A.~Hiley and B.A. Julstrom.
\newblock The quadratic multiple knapsack problem and three heuristic
  approaches to it.
\newblock In P.~Heggernes, editor, {\em Genetic and Evolutionary Computation
  Conference (GECCO)}, pages 547--552, 2006.

\bibitem{HojnyPfetsch}
C.~Hojny and M.E. Pfetsch.
\newblock Handling symmetries in mixed-integer semidefinite programs.
\newblock In {\em International Conference on Integration of Constraint
  Programming, Artificial Intelligence, and Operations Research}, pages 69--78.
  Springer, 2023.

\bibitem{KelrkOliveiraPasechnik}
E.~{\noopsort{Klerk}}{de Klerk}, F.M.~Oliveira Filho, and D.V. Pasechnik.
\newblock {Relaxations of combinatorial problems via association schemes}.
\newblock In Miguel~F. Anjos and Jean~B. Lasserre, editors, {\em {Handbook on
  Semidefinite, Conic and Polynomial Optimization}}, International Series in
  Operations Research \& Management Science, chapter~0, pages 171--199.
  Springer, June 2012.

\bibitem{deKlerkPasechnikSotirov}
E.~{\noopsort{Klerk}}{de Klerk}, D.V. Pasechnik, and R.~Sotirov.
\newblock On semidefinite programming relaxations of the traveling salesman
  problem.
\newblock {\em SIAM J. Optim.}, 19(4):1559--1573, 2009.

\bibitem{KobayashiTakano}
K.~Kobayashi and Y.~Takano.
\newblock A branch-and-cut algorithm for solving mixed-integer semidefinite
  optimization problems.
\newblock {\em Comput. Optim. Appl.}, 75:493--513, 2020.

\bibitem{KoopmansBeckmann}
T.~C. Koopmans and M.~Beckmann.
\newblock Assignment problems and the location of economic activities.
\newblock {\em Econometrica}, 25(1):53--76, 1957.

\bibitem{Kocvara2010}
M.~Ko\v{c}vara.
\newblock Truss topology design with integer variables made easy.
\newblock Optimization Online:
  \url{https://optimization-online.org/2010/05/2614/}, 2010.

\bibitem{KrislockEtAl}
N.~Krislock, J.~Malick, and F.~Roupin.
\newblock A semidefinite branch-and-bound method for solving binary quadratic
  problems.
\newblock {\em ACM Trans. Math. Softw.}, 43(4), 2017.

\bibitem{KuryatnikovaEtAl}
O.~Kuryatnikova, R.~Sotirov, and J.C. Vera.
\newblock The maximum $k$-colorable subgraph problem and related problems.
\newblock {\em INFORMS J. Comput.}, 34(1):656--669, 2021.

\bibitem{LaurentPoljakRendl}
M.~Laurent, S.~Poljak, and F.~Rendl.
\newblock Connections between semidefinite relaxations of the max-cut and
  stable set problems.
\newblock {\em Math. Program.}, 77:225--246, 1997.

\bibitem{LeeScheme}
S.L. Lee, Y.L. Luo, B.E. Sagan, and Y.N. Yeh.
\newblock Eigenvector and eigenvalues of some special graphs. {IV}. multilevel
  circulants.
\newblock {\em Int. J. Quantum Chem.}, 41(1):105--116, 1992.

\bibitem{LetchfordSorensen}
A.N. Letchford and M.M. S{\o}rensen.
\newblock Binary positive semidefinite matrices and associated integer
  polytopes.
\newblock {\em Math. Program. Series A}, 131:253--271, 2012.

\bibitem{LiederEtAl}
F.~Lieder, F.B.A. Rad, and F.~Jarre.
\newblock Unifying semidefinite and set-copositive relaxations of binary
  problems and randomization techniques.
\newblock {\em Comput. Optim. Appl.}, 61:669--688, 2015.

\bibitem{LubinEtAl}
M.~Lubin, J.P. Vielma, and I.~Zadik.
\newblock Mixed-integer convex representability.
\newblock arXiv:1706.05135v3, 2020.

\bibitem{Mars}
S.~Mars.
\newblock {\em Mixed-integer semidefinite programming with an application to
  truss topology design}.
\newblock PhD thesis, FAU Erlangen-Nürnberg, 2013.

\bibitem{MatterPfetsch}
F.~Matter and M.E. Pfetsch.
\newblock Presolving for mixed-integer semidefinite optimization.
\newblock {\em INFORMS J. Optim.}, 5(2):131--154, 2022.

\bibitem{deMeijerSotirovCG}
F.~{\noopsort{Meijer}}{de Meijer} and R.~Sotirov.
\newblock The {C}hv\'atal-{G}omory procedure for integer {SDP}s with
  applications in combinatorial optimization.
\newblock arXiv:2201.10224v2, 2023.

\bibitem{Meurdesoif2005}
P.~Meurdesoif.
\newblock Strengthening the {L}ov\'asz bound for graph coloring.
\newblock {\em Math. Program.}, 102:577 -- 588, 2005.

\bibitem{Narasimhan}
G.~Narasimhan.
\newblock {\em The maximum $k$-colorable subgraph problem}.
\newblock PhD thesis, University of Wisconsin-Madison, 1989.

\bibitem{OostenEtAl}
M.~Oosten, J.H.G.C. Rutten, and F.C.R. Spieksma.
\newblock The clique partitioning problem: Facets and patching facets.
\newblock {\em Networks}, 38:209--226, 2001.

\bibitem{Padberg}
M.~Padberg.
\newblock The {B}oolean quadric polytope: some characteristics, facets and
  relatives.
\newblock {\em Math. Program.}, 45:139--172, 1989.

\bibitem{PiaZhou}
A.~Del Pia and D.~Zhou.
\newblock An {SDP} relaxation for the sparse integer least square problem.
\newblock arXiv:2203.02607, 2023.

\bibitem{PWE:15}
M.~Pilanci, M.J. Wainwright, and L.~{El Ghaoui}.
\newblock Sparse learning via {B}oolean relaxations.
\newblock {\em Math. Prog.}, 151:63--87, 2015.

\bibitem{RechtEtAl}
B.~Recht, M.~Fazel, and P.A. Parrilo.
\newblock Guaranteed minimum-rank solutions of linear matrix equations via
  nuclear norm minimization.
\newblock {\em SIAM Rev.}, 52(3):471--501, 2010.

\bibitem{RendlEtAl}
F.~Rendl, G.~Rinaldi, and A.~Wiegele.
\newblock Solving {m}ax-{c}ut to optimality by intersecting semidefinite and
  polyhedral relaxations.
\newblock {\em Math. Program.}, 121(2):307, 2010.

\bibitem{SchrijverTheta}
A.~Schrijver.
\newblock A comparison of the {D}elsarte and {L}ov\'asz bounds.
\newblock {\em IEEE Trans. Inform. Theory}, 25:425--429, 1979.

\bibitem{Schrijver2005NewCU}
A.~Schrijver.
\newblock New code upper bounds from the {T}erwilliger algebra and semidefinite
  programming.
\newblock {\em IEEE Transactions on Information Theory}, 51:2859--2866, 2005.

\bibitem{wei2022convex}
L.~Wei, A.~Atamtürk, A.~G\'omez, and S.~K\"uç\"ukyavuz.
\newblock On the convex hull of convex quadratic optimization problems with
  indicators.
\newblock arXiv:2201.00387v2, 2022.

\bibitem{YonekuraKanno}
K.~Yonekura and Y.~Kanno.
\newblock Global optimization of robust truss topology via mixed integer
  semidefinite programming.
\newblock {\em Optim. Eng.}, 11(3):355--379, 2010.

\bibitem{ZhengEtAl}
X.~Zheng, H.~Chen, Y.~Xu, Z.~Li, Z.~Lin, and Z.~Liang.
\newblock A mixed-integer {SDP} solution to distributionally robust unit
  commitment with second order moment constraints.
\newblock {\em CSEE J. Power Energy Syst.}, 6(2):374--383, 2020.

\end{thebibliography}
}
\end{document}